\documentclass[11 pt, twoside]{amsart}
\usepackage[a4paper,margin=0.75in]{geometry}
\usepackage{amssymb,latexsym,amsthm,amsmath,mathtools,amsfonts,hyperref,fancyhdr, comment,enumitem, cleveref,xcolor,bbm, caption, amssymb, mathrsfs, bm, tikz, xcolor}
\usepackage{mathtools}
\usepackage[enableskew]{youngtab}
\usepackage[centertableaux]{ytableau}

\hypersetup{
	colorlinks=true,
	linkcolor=blue,
	filecolor=magenta,
	urlcolor=cyan,
	citecolor=red,
}

\long\def\symbolfootnote[#1]#2{\begingroup%
	\def\thefootnote{\fnsymbol{footnote}}\footnote[#1]{#2}\endgroup}

\newcommand{\C}{\ensuremath{\mathscr{C}}}
\newcommand{\D}{\ensuremath{\mathscr{D}}}
\newcommand{\p}{\mathscr{P}}

\newcommand{\x}{\mathfrak X}
\newcommand{\h}{\textup{ht}}

\newcommand{\DP}{\mathscr{DP}}
\newcommand{\OP}{\mathscr{OP}}
\newcommand{\wtilde}{\widetilde}

\newcommand{\tensor}{\otimes}
\newcommand{\Mod}{\mathrm{mod}}

\newcommand{\ind}{\mathrm{Ind}}

\newcommand{\irr}{\mathrm{Irr}}

\newcommand{\och}{\mathrm{och}}
\newcommand{\ty}{\mathrm{ty}}

\makeatletter
\def\imod#1{\allowbreak\mkern10mu({\operator@font mod}\,\,#1)}
\makeatother

\newtheorem{theorem}{Theorem}[section]
\newtheorem{lemma}[theorem]{Lemma}
\newtheorem{corollary}[theorem]{Corollary}
\newtheorem{proposition}[theorem]{Proposition}
\newtheorem*{theorem*}{Theorem}
\newtheorem{definition}[theorem]{Definition}

\newtheorem{remark}[theorem]{Remark}

\newtheorem{example}[theorem]{Example}

\numberwithin{equation}{section}
\newcommand{\ignore}[1]{}

\newcommand{\mynote}[1]{}

\newcommand{\bx}[1]{\fcolorbox{red}{white}{$#1$}}

\newcommand{\bardiag}[3]{%
\begin{tikzpicture}[scale=0.5,baseline=(current bounding box.center)]
  \foreach \rx/\rr in {#3}{%
    \pgfmathtruncatemacro\ry{#1-\rr}
    \fill[red!75!orange] (\rx,\ry) rectangle (\rx+1,\ry+1);
  }
  \xdef\prevE{0}
  \foreach \len [count=\ri] in {#2}{%
    \pgfmathtruncatemacro\yy{#1-\ri}
    \pgfmathtruncatemacro\st{\ri-1}
    \pgfmathtruncatemacro\ee{\st+\len}
    \pgfmathtruncatemacro\mm{\len-1}
    \foreach \xx in {0,...,\mm}{\draw (\st+\xx,\yy) rectangle (\st+\xx+1,\yy+1);}
    \node[anchor=west,inner sep=1pt] at (\ee,\yy+0.5) {$0$};
    \pgfmathtruncatemacro\dd{\prevE-\ee}
    \ifnum\dd>0
      \foreach \jj in {1,...,\dd}{%
        \node[anchor=north,inner sep=1pt] at (\prevE+0.8-\jj,\yy+1) {$1$};}
    \fi
    \xdef\prevE{\ee}%
  }
  \pgfmathtruncatemacro\dd{\prevE-#1+1}
  \foreach \jj in {1,...,\dd}{%
    \node[anchor=north,inner sep=1pt] at (\prevE+0.8-\jj,0) {$1$};}
\end{tikzpicture}}

\newcommand{\arr}[1]{\xrightarrow{\ \text{\scriptsize #1}\ }}

\newcommand{\tm}[2]{\tikz[remember picture,baseline=(#1.base)]\node[inner sep=0pt](#1){#2};}
\newcommand{\flowarrow}[4]{%
\begin{tikzpicture}[remember picture,overlay]
  \draw[->,thick] (#1.south) -- ++(0,-0.7) coordinate(m) -- (m -| #2.north) -- (#2.north);
  \path (m) -- (m -| #2.north) node[midway,above,font=\scriptsize]{#3} node[midway,below,font=\scriptsize]{#4};
\end{tikzpicture}}

\begin{document}

\title[A Morris recursion rule for wreath products] {A Morris recursion rule for values of the spin characters of wreath products}

\author{Rijubrata Kundu}
\address{Department of Mathematics, Birla Institute of Technology and Science, Pilani, Rajasthan, 333031, India}
\email{rijubrata8@gmail.com, rijubrata.kundu@pilani.bits-pilani.ac.in}
	
\author{Papi Ray}
\address{Department of Mathematical Sciences, Indian Institute of Science Education and Research, Mohali-140306, India}
\email{popiroy93@gmail.com}
	
\date{\today}
\subjclass[2010]{20C30, 20E22, 05E10}
\keywords{spin wreath products, spin characters, Murnaghan-Nakayama rule, Schur Q-functions}

\begin{abstract}
    The classical Murnaghan-Nakayama rule is a recursive formula for computing the values of complex irreducible characters of the symmetric group $S_n$. Alun Morris proved a recursive formula for evaluating the values of irreducible spin characters of $\wtilde{S}_n$, where $\wtilde{S}_n$ is one of the Schur covers of $S_n$  defined by $\wtilde{S}_n:=\langle t_1,t_2,\cdots,t_{n-1},z\ |\ z^2=1,\ t_i^2=z,\ (t_it_{i+1})^3=z, \ t_it_j=zt_jt_i\ \text{if}\ |i-j|>1\rangle$. A recursive formula for evaluating the values of complex irreducible characters of the wreath product $G\wr S_n$, where $G$ is a finite group, was proved by J. Stembridge. In this article, we state and prove a recursive formula to compute the values of the irreducible spin characters of the wreath product $G\wr \wtilde{S}_n$. For the convenience of implementing these recursive formulas, we extend the notion of 0-1 boundary sequence of a Young diagram to shifted Young diagrams. 
\end{abstract}

\maketitle

\section{Introduction}\label{intro}
    Let $\wtilde{S}_n$ be one of the Schur covers of $S_n$ defined by $\wtilde{S}_n:=\langle t_1,t_2,\cdots,t_{n-1},z\ |\ z^2=1,\ t_i^2=z,\ (t_it_{i+1})^3=z, \ t_it_j=zt_jt_i\ \text{if}\ |i-j|>1\rangle$. In 1911, Schur developed the theory of complex irreducible characters of $\wtilde{S}_n$ (equivalently, complex projective characters of $S_n$). The irreducible characters of $\wtilde{S}_n$ are of two kinds: (1) the irreducible characters that arise from the irreducible characters of $S_n$ (these are called positive irreducible characters as they are characterized by the fact that the corresponding representations map $z$ to $I$), and (2) the irreducible spin characters, that is, characters whose corresponding representations map $z$ to $-I$ (these are also called negative irreducible characters). In analogy to the fact that irreducible characters of $S_n$ correspond to Schur functions, Schur connected the irreducible spin characters to certain symmetric functions, which are now well-known as Schur Q-functions. Much later, Morris used these Schur Q-functions to develop a Murnaghan-Nakayama-type recursive rule to evaluate values of irreducible spin characters on conjugacy classes indexed by partitions having only odd parts (see \cite{mo.64}). In the process, he introduced a bar graph of partitions to interpret the recursive rule. Macdonald introduced shifted Young diagrams to study the combinatorics of Schur Q-functions (see \cite{mac}). Sagan and Worley (see \cite{sa,wo}) independently developed the combinatorial theory of shifted diagrams and tableaux. Further, Stembridge (see \cite{stem}) gave a symmetric function approach to construct the irreducible spin characters while also obtaining several new results, such as a shifted analog of the Littlewood-Richardson rule and results on the Kronecker product of Schur Q-functions (equivalently, the product of irreducible spin characters). A detailed account of Schur's original work and many subsequent developments, including those mentioned above, is given in the book by Hoffman and Humphreys (see \cite{hoff}).

    \medskip

    The representation theory of the wreath product $G\wr S_n$, where $G$ is a finite group, was developed by Specht in 1932. A representation-theoretic way of obtaining the complex irreducible characters of $G\wr S_n$ using Clifford theory is given in \cite[Chapter 4]{jk}. Macdonald (see \cite[Appendix B]{mac}) developed a symmetric function approach to obtain the irreducible characters of $G\wr S_n$ by defining a Frobenius characteristic map from the space of class functions of $G\wr S_n$ onto a space of homogeneous degree $n$ symmetric functions parametrized by conjugacy classes of $G$ (see Section~\ref{Frob_char_GwrSntilde}). Stembridge (see \cite[Appendix]{stem1}) gave a  Murnaghan-Nakayama-type recursive rule to evaluate values of irreducible characters of $G\wr S_n$.

    \medskip

    The representation theory of the group $G\wr \wtilde{S}_n$, where $G$ is a finite group, was studied in \cite{hoff1} from the viewpoint of Hopf algebras. Alike $\wtilde{S}_n$, there are two kinds of complex irreducible characters of $G\wr \wtilde{S}_n$, namely the positive irreducible characters which arise from the irreducible characters of $G\wr S_n$ and the irreducible spin characters (or negative irreducible characters). In \cite{fren}, observing deep connections with the McKay correspondence,  Frenkel, Jing, and Wang developed a vertex-algebra approach to the spin representation theory of $G\wr\wtilde{S}_n$.  The values of the irreducible spin characters of $G\wr \wtilde{S}_n$ evaluated on conjugacy classes indexed by partition-valued functions with odd parts were obtained as matrix coefficients of certain products of twisted vertex operators (\cite[Theorem 9.7]{fren}). Moreover, analogous to  Macdonald's symmetric function approach to the character theory of $G\wr S_n$, they introduced an appropriate space of symmetric functions and established, via a Frobenius characteristic map, an isomorphism between this space and the space of spin class functions of $G\wr\widetilde{S}_n$. The resulting symmetric-function space admits a distinguished basis given by generalized Schur Q-functions, which correspond to the irreducible spin characters of $G\wr\widetilde{S}_n$. 

    \medskip

    While the Morris recursion rule provides a method for computing values of the irreducible spin characters of $\widetilde{S}_n$ on conjugacy classes indexed by partitions with only odd parts, explicit formulas for their values on conjugacy classes indexed by odd strict partitions have been determined (see \cite[Theorem 8.7]{hoff}). Hu and Jing explicitly determined the values of irreducible spin characters of $G\wr \wtilde{S}_n$ on conjugacy classes indexed by odd strict partition-valued functions for a finite abelian group $G$ in \cite{hu} and then generalized the same for all finite groups in \cite{hu1}.

    \medskip

    The main theorem of this article, Theorem~\ref{main_theorem}, provides a recursion rule to compute the values of the irreducible spin characters of $G\wr \wtilde{S}_n$ on conjugacy classes indexed by partition-valued functions with only odd parts. This generalizes the Morris recursion rule. As expected, the rule looks very similar to Stembridge's recurrence rule for ordinary wreath products, where the combinatorics of Young diagrams is replaced by that of shifted diagrams.  We prove our main theorem using the symmetric function framework developed in \cite{fren} for irreducible spin characters of $G\wr \wtilde{S}_n$. Inspired by the well-known $0$-$1$ boundary sequence for ordinary Young diagrams, we introduce its analogue for shifted Young diagrams and use this description to interpret the Morris recursion rule (see Section~\ref{0-1-seq-shifted}). This also gives a convenient way to implement the recursion rule for  $\wtilde{S}_n$ as well as for $G\wr \wtilde{S}_n$. The article is organized as follows: In Section~\ref{irrep_tildeS_n}, we briefly discuss the theory of irreducible spin characters of $\wtilde{S}_n$ and the Morris recursion, which we use in the proof of our main theorem. In Section~\ref{irreps_of_GwrSntilde}, we discuss the construction of irreducible spin representations of $G\wr \wtilde{S}_n$, followed by Section~\ref{Frob_char_GwrSntilde}, where the symmetric function framework for the irreducible spin characters of $G\wr \wtilde{S}_n$ is discussed following \cite{fren}. We state and prove our main theorem in Section~\ref{sec_main_theorem}. We conclude our article with Section~\ref{0-1-seq-shifted}, where we discuss some aspects of $0$-$1$ boundary sequence of a shifted diagram.
    
\section{Complex Irreducible Representation of \texorpdfstring{$\wtilde{S}_n$}{Sn tilde}}\label{irrep_tildeS_n}
    In this section, we discuss the theory of irreducible spin characters of $\wtilde{S}_n$ briefly. We divide this into several subsections, starting with a description of conjugacy classes of $\wtilde{S}_n$, followed by the definition of irreducible spin characters (or negative irreducible characters) of $\wtilde{S}_n$ and the Frobenius characteristic function that eventually connects the irreducible spin characters to Schur $Q$-functions. We closely follow \cite{hoff, stem} for this purpose. In the final important subsection, following \cite{mac}, we discuss the Morris recursion rule for evaluating values of irreducible spin characters of $\wtilde{S}_n$. 

    \medskip
    
    Let $\wtilde{S}_n$ be one of the Schur covers of $S_n$ defined by $\wtilde{S}_n:=\langle t_1,t_2,\cdots,t_{n-1},z\ |\ z^2=1,\ t_i^2=z,\ (t_it_{i+1})^3=z, \ t_it_j=zt_jt_i\ \text{if}\ |i-j|>1\rangle$. We recall that $S_n=\langle s_1,s_2,\cdots,s_{n-1}\ |\ s_i^2=e,\ (s_is_{i+1})^3=e,\  s_is_j=s_js_i,\ \mbox{if}\ |i-j|>1 \rangle$, where $s_i=(i,i+1)\in S_n$ for all $1\leq i\leq n-1$. The canonical homomorphism $\theta: \wtilde{S}_n\to S_n$ is defined by $\theta(z)=1$ and $\theta(t_i)=s_i$. Clearly, $\wtilde{S}_n/\{1,z\}\cong S_n$ (under $\theta$) and we have a short exact sequence $$\{1\}\to \mathbb{Z}/2\mathbb{Z} \xrightarrow{i} \wtilde{S}_n \xrightarrow{\theta}S_n\to \{1\}.$$
    For $\sigma \in \wtilde{S}_n,\ \theta(\sigma)\in S_n$ is called the \textit{projection} of $\sigma$ in $S_n$. More generally, for any subset $X\subseteq \wtilde{S}_n,\ \theta(X)$ is called the \textit{projection} of $X$ in $S_n$. Similarly, for any $Y\subseteq S_n,\ \theta^{-1}(Y)$ is called the \textit{pull-back} of $Y$ in $\wtilde{S}_n$. We note that if $\sigma \in \wtilde{S}_n$ is such that $\theta(\sigma)=\pi$, then $\theta^{-1}(\pi)=\{\sigma, z\sigma\}$. For $\sigma \in \wtilde{S}_n$, let $d(\sigma)$ denote the parity of $\theta(\sigma)\in S_n$, that is, $d(\sigma)=0$ if $\theta(\sigma)$ is an even permutation and $d(\sigma)=1$, if $\theta(\sigma)$ is an odd permutation.

 
    \subsection{Conjugacy classes of \texorpdfstring{$\wtilde{S}_n$}{Sn tilde}}\label{conj_classes_Sntilde}
        Let $\lambda=(\lambda_1,\lambda_2,\cdots,\lambda_l)$ with $\lambda_1\geq \lambda_2\geq\cdots\geq \lambda_l\geq 0$ be a partition of $n$ (written $\lambda \vdash n$). It is sometimes convenient to write $\lambda =\langle 1^{m_1},2^{m_2},\cdots \rangle$, where $m_i$ denotes the number of times $i$ occurs as part of $\lambda$. For $\pi \in S_n$, let $\ty(\pi)=\langle 1^{m_1},2^{m_2},\cdots \rangle$, where $m_i$ denotes the number of $i$-cycles occurring in a disjoint cycle decomposition of $\pi$. Thus, $\ty(\pi)$ is a partition of $n$, and it is well known that if $\pi, \sigma\in S_n$, then $\pi$ and $\sigma$ are conjugate if and only if $\ty(\pi)=\ty(\sigma)$, Thus, conjugacy classes of $S_n$ are indexed by partitions of $n$. Let $\C_{\lambda}$ be the conjugacy class of $S_n$ indexed by $\lambda$. Let $l(\lambda)$ denote the length of $\lambda$, that is, the number of parts in $\lambda$. Recall that $\pi \in S_n$ is odd if and only if $n-l(\lambda)$ is odd, where $\lambda=\ty(\pi)$. Consequently, we say that a partition $\lambda\vdash n$ is odd if $n-l(\lambda)$ is odd and even if $n-l(\lambda)$ is even. We define the parity of $\lambda$ as $n-l(\lambda)\ (\Mod\; 2)$, and denote it by $d(\lambda)$. It is easy to see that the pull-back of $\C_\lambda$ in $\wtilde{S}_n$, denoted by $\D_{\lambda}$, is either a conjugacy class of $\wtilde{S}_n$ (in which case we call it a non-split class) or it splits into two conjugacy classes, say $\D_\lambda^+$ and $\D_\lambda^-$, and is related by $\D_\lambda^-=z\D_\lambda^+$. We call $\D_{\lambda}^{+}$ and $\D_{\lambda}^{-}$ split classes of $\wtilde{S}_n$. 

        \begin{theorem}\cite[Theorem 3.8]{hoff}\label{split_classes_Sntilde}
            Let $\lambda \vdash n$. The pull-back of $\C_\lambda$ in $\wtilde{S}_n$ splits into two conjugacy classes $\D_\lambda^+$ and $\D_\lambda^-$ if and only if one of the following holds.
            \begin{enumerate}
                \item $\lambda$ has only odd parts, 
                \item $\lambda$ has distinct parts and $n-l(\lambda)$ is odd.
            \end{enumerate}
        \end{theorem}

        Let $\p$ denote the set of all partitions (including the empty partition of zero). If $\lambda\in \p$, $|\lambda|$ denotes the sum of its parts. Let $\OP$ denote the set of all partitions with only odd parts. Let $\DP$ denote the set of all strict partitions. For $n\in \mathbb{N}$, let $\p_n$ denote the set of partitions of $n$. Set $\OP_n:=\OP\cap \p_n$ and $\DP_n:=\DP\cap \p_n$. For $i=0,1$, let $\DP_n^i=\{\lambda \in \DP_n\ |\ d(\lambda)=i\}$. With these notations, Theorem \ref{split_classes_Sntilde} asserts that if $\lambda\in \OP_n \cup \DP_n^1$, $\D_{\lambda}$ splits into two conjugacy classes $\D_\lambda^+$ and $\D_\lambda^-$, otherwise it remains a $\wtilde{S}_n$-conjugacy class. Finally, we label the split classes $\D_\lambda^+$ and $\D_\lambda^-$ in the following way (see \cite[2.1]{stem}): 
        If $\lambda=(\lambda_1,\lambda_2,\cdots,\lambda_l)$, define $\sigma_\lambda:=\pi_1\pi_2\cdots \pi_l$, where $\pi_j:=t_{a+1}t_{a+2}\cdots t_{a+\lambda_j-1}$, $a=\lambda_1+\lambda_2+\cdots+\lambda_{j-1}$, and designate the conjugacy class of $\sigma_\lambda$ as $\D_\lambda^+$. Thus, $\D_\lambda^-$ is the conjugacy class of $z\sigma_\lambda$. The number of conjugacy classes of $\wtilde{S}_n$ equals $|\p_n|+|\OP_n|+|\DP_n^1|$. 

        \subsection{Irreducible spin characters of \texorpdfstring{$\wtilde{S}_n$}{Sn tilde}}\label{spin_char_Sntilde}
        Let $\varphi$ be an irreducible representation of $\wtilde{S}_n$. By Schur's lemma, $\varphi(z)=\pm I$. An irreducible representation $\varphi$ of $\wtilde{S}_n$ is called \textit{positive} if $\varphi(z)=I$ and is called \textit{negative} or \textit{spin representation} if $\varphi(z)=-I$. The corresponding character is called a \textit{irreducible spin character}. The positive irreducible representations of $\wtilde{S}_n$ are lifts of the irreducible representations of $S_n$ (via the map $\theta$), which are indexed by partitions of $n$. If $\varphi$ is a representation of $\wtilde{S}_n$, then $\varphi^a(\pi)=(-1)^{d(\pi)}\varphi(\pi)$ is a representation of $\wtilde{S}_n$, called the \textit{associate representation} of $\varphi$. $\varphi$ is irreducible if and only if $\varphi^a$ is irreducible. $\varphi$ is called \textit{double} or \textit{self-associate} if $\varphi$ and $\varphi^a$ are isomorphic. Since the number of positive irreducible characters of $\wtilde{S}_n$ is $|\p_n|$, it is clear that the number of irreducible spin characters is $|\OP_n|+|\DP_n^1|$ which is same as $|\DP_n|+|\DP_n^1|$. For each $\lambda \in \DP_n$, there is an irreducible spin representation of $\wtilde{S}_n$, which we denote by $\psi_\lambda$. If $\lambda \in \DP_n^0$, $\psi_\lambda$ is self-associate. If  $\lambda\in \DP_n^1$,  then $\psi_\lambda \neq \psi^a_\lambda$, whence $\psi_\lambda^a$ yields the remaining $|\DP_n^1|$ irreducible spin representations. Let $\wtilde{\chi}_{\lambda}$ denote the character of the irreducible spin representation $\psi_{\lambda}$, where $\lambda \in \DP_n$.  

        \medskip

        It is easy to see that $\wtilde{\chi}_{\lambda}(\D_\mu)=0$ for any non-split class $\D_\mu$, while $\wtilde{\chi}_{\lambda}(\D_\mu^+)=-\wtilde{\chi}_{\lambda}(\D_\mu^-)$ for split classes $\D_\mu^+$ and $\D_\mu^-$. For $\lambda=(n)$, we recall that $\wtilde{\chi}_{(n)}$ is the \textit{basic character} of $\wtilde{S}_n$, which we simply denote by $\wtilde{\chi}_n$. The character values of $\wtilde{\chi}_n$ are explicitly given in \cite[Theorem 3.3]{stem}. From the theorem, we also observe that the $\wtilde{\chi}_n(\sigma_{\mu})$ is positive whenever $\mu \in \OP_n$. 



    \subsection{Frobenius characteristic function}\label{Frob_char_Sntilde}
        Let $R(\wtilde{S}_n)$ be the space of all \textit{class functions} of $\wtilde{S}_n$ over $\mathbb{C}$. Then $R(\wtilde{S}_n)=R^+(\wtilde{S}_n)\oplus R^-(\wtilde{S}_n)$, where $R^+(\wtilde{S}_n)$ and $R^-(\wtilde{S}_n)$ are subspaces spanned by the positive irreducible characters and the irreducible spin characters, respectively. For $\lambda\in \OP_n\cup \DP_n^1$, let $1'_\lambda$ be defined by
        \[{1'_{\lambda}}(x)=\begin{cases}
        1 & \text{if} \quad x \in \D_{\lambda}^+,\\
        -1 & \text{if}\quad x\in \D_{\lambda}^-,\\
        0 & \text{otherwise}.
        \end{cases}\]
        Then $ R^-(\wtilde{S}_n)=\langle 1'_\lambda\ |\ \lambda \in \OP_n\cup \DP_n^1\rangle$. $R^-(\wtilde{S}_n)$ is called the space of \textit{spin class functions} of $\wtilde{S}_n$. Define $R^0(\wtilde{S}_n)=\langle 1'_\lambda\ |\ \lambda \in \OP_n \rangle \subseteq R^-(\wtilde{S}_n)$. $R^0(\wtilde{S}_n)$ is called the space of \textit{super spin class functions}. We note that $\wtilde{\chi}_\lambda \in R^0(\wtilde{S}_n)$ when $\lambda \in \DP_n^0$ and $\wtilde{\chi}_\lambda+\wtilde{\chi}_\lambda^a \in R^0(\wtilde{S}_n)$ when $\lambda \in \DP_n^1$.

        \medskip

        Let $\Lambda_n$ be the space of homogeneous symmetric functions over $\mathbb{C}$ of degree $n$. Let $\displaystyle \Lambda:=\oplus_{n=0}^\infty \Lambda_n$. Let $p_r$ denote the power sum symmetric function of degree $r$. Let $\Delta$ be the subalgebra of $\Lambda$ generated by $\{1,p_1,p_3,p_5,\cdots\}$. Let $\Delta_n:=\langle p_\lambda\ |\ \lambda \in \OP_n \rangle$, where $p_{\lambda}$ is the power sum symmetric function corresponding to $\lambda$. Clearly, $\Delta_n$ is a subspace of $\Lambda_n$. Moreover, $\displaystyle \Delta=\oplus_{n=0}^\infty \Delta_n$. For $n\in \mathbb{N}$, the symmetric functions  $q_n$ are defined by the generating function $$1+\sum_{n=1}^\infty q_nt^n=\prod_{i \geq 1} \frac{1+x_it}{1-x_it}.$$
        This yields $q_n=\sum_{\lambda \vdash n}2^{l(\lambda)}m_\lambda,$ where $m_{\lambda}$ denotes the monomial symmetric function corresponding to $\lambda$. The following identity connects these with the power symmetric functions.
        $$q_n=\sum_{\lambda \in \OP_n}\frac{2^{l(\lambda)}p_\lambda}{z_\lambda},$$
        where $z_{\lambda}$ is the size of the centralizer of an element of the conjugacy class $\C_{\lambda}$ of $S_n$. Thus, $q_n\in \Delta_n$. For $\lambda=(\lambda_1,\ldots,\lambda_l)\vdash n$, set $q_{\lambda}:=q_{\lambda_1}\cdots q_{\lambda_l}$. Then $q_{\lambda}\in \Delta_n$. Moreover, $\Delta_n=\langle q_\lambda\ |\ \lambda\in \DP_n \rangle$. Define an inner product on $\Delta_n$ as follows: For $\lambda,\mu\in \OP_n$,
        $$\langle p_\lambda, p_\mu \rangle=2^{-l(\lambda)}z_\lambda\delta_{\lambda \mu}.$$

        \noindent The Frobenius characteristic function is a linear map $\och_n:R^-(\wtilde{S}_n)\to \Delta_n$ defined by \[\och_n(1'_{\lambda})=\begin{cases}
        \dfrac{2^{\frac{l(\lambda)}{2}}}{z_\lambda}p_\lambda & \text{if} \quad \lambda \in \OP_n,\\
        0 & \text{if} \quad \lambda \in \DP^-_n.
        \end{cases}\]

        \noindent Thus, for $f \in R^-(\wtilde{S}_n)$, $$\och_n(f)=\sum_{\lambda \in \OP_n}\frac{f(\sigma_\lambda)2^{l(\lambda)/2}}{z_\lambda}p_\lambda.$$
        
        \noindent The map $\och|_{R^0(\wtilde{S}_n)}$ is also an isometry from $R^0(\wtilde{S}_n)$ onto $\Delta_n$, that is, for $f,g \in R^0(\wtilde{S}_n)$, 
        $$\langle \och_n(f),\och_n(g) \rangle_{\Delta_n}=\langle f,g \rangle_{\wtilde{S}_n}.$$

        \noindent More generally, for $f\in R^-(\wtilde{S}_n)$ and $g\in R^0(\wtilde{S}_n)$,
        $\langle \och_n(f),\och_n(g) \rangle=\langle f,g \rangle_{\wtilde{S}_n}.$ We note the obvious fact that och$_n(f)=\och_n(f^a)$. The images of the irreducible spin characters under $\och_n$ are given by Schur $Q$-functions, which we define now. We define it combinatorially (see \cite[Section 6]{stem}) using the notation of shifted tableaux. For any $\lambda \in \DP_n$, the associated shifted diagram is denoted by ${SD}_\lambda$ and is defined by $${SD}_\lambda=\{(i,j)\in \mathbb{N}^2:\ i \leq j\leq\lambda_j+i-1,\ 1 \leq i \leq l(\lambda)\}.$$
        Let $P=\{1',1,2',2,\cdots\}$, where $1'<1<2'<2<\cdots$. The letters $1',2',\cdots$ are said to be \textit{marked}. For any $a'\in P$, the unmarked version of $a'$ is denoted by $a$. For any $\lambda \in \DP_n$, a \textit{semi-standard shifted tableau} $T$ of shape $\lambda$ is a map $T:{SD}_\lambda \to P$ such that (1) $T(i,j) \leq T(i+1,j),\ T(i,j)\leq T(i,j+1)$,\;(2) Each column has at most one $k\ (k=1,2,\cdots)$,\;(3) Each row has at most one $k'\ (k'=1',2',\cdots)$. The set of all semi-standard shifted tableau (abbrv. SSST) of shape $\lambda$ is denoted by SSST$_\lambda$.

        \begin{example}
            For $\lambda=(7,6,3,1)\in \DP_{17}$,
            $$T=\hspace{.2cm}\ytableausetup{smalltableaux}\begin{ytableau}
            1' & 1 & 1 & 2' & 4' & 4 & 6 \\
            \none & 2' & 2 & 3 & 4' & 5' & 7' \\
            \none & \none & 3' & 4' & 4 \\
            \none & \none & \none & 7 
            \end{ytableau} ,$$
            is a semi-standard shifted tableau of shape $\lambda$ on $P$.
        \end{example}
         
         \noindent For any $\lambda \in \DP_n$ and $\mu \in \p_n$, we say $T$ is a SSST of shape $\lambda$ and type $\mu$ if $i$ and $i'$ together occur $\mu_i$ times in $T$. With these notations, for $\lambda\in \DP_n$, we define the Schur's $Q$-functions as follows:
        $$Q_\lambda:=\sum_{T\in \mbox{SSST}_\lambda}x^T.$$
        Let $k'_{\lambda \mu}$ denote the number of SSST of shape $\lambda$ and type $\mu$. Then
        $$Q_\lambda=\sum_{\mu \vdash n}k'_{\lambda \mu}m_\mu.$$
        From \cite[Corollary 6.2]{stem}, it is also known that
        $$q_\lambda=\sum_{\lambda \in \DP_n}\frac{k'_{\lambda \mu}}{2^{l(\lambda)}}Q_\lambda.$$
        and $\Delta_n=\langle Q_\lambda\ |\ \lambda \in \DP_n \rangle$. Moreover, $\langle Q_\lambda, Q_\mu \rangle=2^{l(\lambda)}\delta_{\lambda \mu}$. Define $Q_\lambda^*=2^{\frac{-l(\lambda)}{2}}Q_\lambda$. Thus $\{Q^{*}_\lambda\ |\ \lambda \in \DP_n\}$ is an orthonormal basis of $\Delta_n$. The following theorem can be found in \cite[Section 7]{stem} (see also \cite[Theorem 8.6]{hoff}).

        \begin{theorem}\label{image_spin_char_S_ntilde}
            For $\lambda \in \DP_n$ , $\displaystyle \och_n(\wtilde{\chi}_\lambda)=2^{\frac{-d(\lambda)}{2}}Q_\lambda^*$.
        \end{theorem}


        \noindent When $\lambda\in \DP_n$ and $\mu \in \OP_n$, we have the analogue of Murnaghan-Nakayama rule:
        \begin{equation}\label{spin_char_value_inn_prod}
        \wtilde{\chi}_\lambda(\sigma_\mu)=2^{-\frac{d(\lambda)}{2}}\langle Q_\lambda^* , 2^{\frac{l(\mu)}{2}}p_\mu \rangle=2^{\frac{l(\mu)-l(\lambda)-d(\lambda)}{2}}\langle Q_\lambda, p_\mu \rangle.
        \end{equation}
        Similar to the  Murnaghan-Nakayama rule for computing the values of the irreducible characters of $S_n$, a rule for evaluating the values of the irreducible spin characters of $\wtilde{S}_n$ was given by Morris, which we discuss next.

    \subsection{Morris Recursion Rule}\label{Morris_recursion_rule_Sntilde}
        The well-known Murnaghan-Nakayama rule is a (recursive) formula to calculate the values of irreducible characters of the symmetric group $S_n$ (see \cite[Theorem 5.11.1]{ap}. In \cite[Theorem 2]{mo.64}, Morris gave a recursive formula to calculate $\wtilde{\chi}_{\lambda}(\sigma_{\mu})$ of the irreducible spin character $\wtilde{\chi}_{\lambda}$ of $\wtilde{S}_n$, where $\lambda \in \DP_n$ and $\mu\in \OP_n$. There are two ways to describe this recursive rule, one in terms of the strict partition $\lambda$ itself and the other in terms of the shifted diagram of $\lambda$. We briefly discuss this following \cite[Page 266]{mac}. For two strict partitions $\lambda,\mu$ such that $SD_{\mu}\subseteq SD_{\mu}$ (where $\subseteq$ refers to diagram containment), let $SD_{\lambda/\mu}$ denote the shifted skew diagram of shape $\lambda/\mu$. A \textit{border strip} (or a rim-hook) is a connected shifted skew diagram possessing no $2\times 2$ square. The height of a border strip $\alpha$ is one less than the number of rows present in it. It is denoted by $\h(\alpha)$. A \textit{double strip} is a shifted skew diagram formed by the union of two border strips that both end on the main diagonal. A double strip $\delta$ can be cut into two non-empty connected pieces, one piece $\alpha$ consisting of the diagonals in $\delta$ of length $2$, and the other piece $\beta$ consisting of the border strip formed by the diagonals of length $1$. The height of the double strip $\delta$, denoted by $\h(\delta)$, is defined to be $\frac{1}{2} |\alpha|+ \h(\beta).$ Let $\varepsilon(\delta)$ take the value $(-1)^{\h(\delta)}$ or $2(-1)^{\h(\delta)}$ according as $\delta$ is a border strip or a double strip. 
        
        \medskip

        \noindent Let $\lambda=(\lambda_1,\lambda_2,\cdots,\lambda_l)\in \DP_n$ and $r$ be a odd positive integer. We consider three types of removals of size $r$ from $\lambda$, each of which yields a strict partition of $n-r$. Each such removal can be interpreted in terms of the removal of a border strip or double strip from  $SD_{\lambda}$.

        \begin{enumerate}
            \item \textbf{Type-I}: there exists $i$ such that $\lambda_i>r$ and $\lambda_1,\ldots,\lambda_i-r, \ldots, \lambda_l$ are all distinct. Let $j$ be such that $\lambda_j>\lambda_i-r>\lambda_{j+1}$ (here we assume $\lambda_{l+1}=0$). Then we obtain the strict partition $\hat{\lambda}=(\lambda_1,\ldots,\lambda_{i-1},\lambda_{i+1},\ldots,\lambda_j,\lambda_i-r,\ldots)$ of $n-r$, by subtracting $r$ from the part $\lambda_i$. In terms of shifted diagrams, $SD_{\hat{\lambda}}$ is obtained from $SD_{\lambda}$ by removing a border strip of size $r$ which starts at the $i$-th row and ends at the $j$-th row. It is clear that $\varepsilon(SD_{\lambda/\hat{\lambda}})=(-1)^{j-i}$.

            \item \textbf{Type-II}: there exists $i$ such that $\lambda_i=r$. Then we obtain the strict partition $\hat{\lambda}=(\lambda_1,\ldots,\lambda_{i-1},\lambda_{i+1},\ldots)$ of $n-r$, by removing the part $\lambda_i$ entirely. In terms of shifted diagrams, $SD_{\hat{\lambda}}$ is obtained from $SD_{\lambda}$ by removing a border strip of size $r$ which starts at the $i$-th row and ends at the main diagonal of the last row. Once again it is clear that $\varepsilon(SD_{\lambda/\hat{\lambda}})=(-1)^{l-i}$.

            \item \textbf{Type-III}: there exists $i,j$ such that $r=\lambda_i+\lambda_j$. Then we obtain the strict partition $\hat{\lambda}=(\lambda_1,\ldots,\lambda_{i-1},\lambda_{i+1},\ldots,\lambda_{j-1},\lambda_{j+1},\ldots)$ of $n-r$, by removing both the parts $\lambda_i$ and $\lambda_j$ entirely. In terms of shifted diagrams, $SD_{\hat{\lambda}}$ is obtained from $SD_{\lambda}$ by the removal of the border strip that starts at the $j$-th row and ends at the main diagonal of the last row (this is just removing the part $\lambda_j$), followed by the removal of the border strip that starts at the $i$-th row and ends at the main diagonal of the last row (this is just removing the part $\lambda_i$). This is precisely the removal of a double strip of size $r$. It readily follows that $\varepsilon(SD_{\lambda/\hat{\lambda}})=(-1)^{j-i-1+\lambda_j}$.
        \end{enumerate}

        \noindent Let $I_{\lambda}(r)$ denote the set of all strict partitions $\hat{\lambda}$ of $n-r$ obtained from $\lambda$ by applying one of the three removal procedures. Now we are ready to state Morris recursion rule. Let $\lambda\in \DP_n$ and $\mu=(\mu_1,\ldots,\mu_k)\in \OP_n$. Then

        \begin{equation}
        \langle Q_{\lambda},p_{\mu}\rangle=\sum_{\hat{\lambda}\in I_{\lambda}(\mu_k)}\varepsilon(SD_{\lambda/\hat{\lambda}})\langle Q_{\hat{\lambda}},p_{\hat{\mu}} \rangle,
        \end{equation}
        where $\hat{\mu}=(\mu_1,\ldots,\mu_{k-1})$. Equivalently,

        \begin{equation}
        \langle Q_{\lambda},p_{\mu}\rangle=\sum_{S=(\lambda^{(0)},\ldots,\lambda^{(k)})}\varepsilon(S),
        \end{equation}
        summed over all sequences of strict partitions $S=(\lambda^{(0)},\ldots,\lambda^{(k)})$ such that $\emptyset=\lambda^{(0)}\subset \lambda^{(1)}\subset \ldots\subset \lambda^{(k)}=\lambda$, $SD_{\lambda^{(i)}/\lambda^{(i-1)}}$ is a border strip or double strip of size $\mu_{i}$ for all $1\leq i\leq k$, and $\varepsilon(S)=\prod_{i=1}^{k} \varepsilon(SD_{\lambda^{(i)}/\lambda^{(i-1)}})$. Using Equation~\ref{spin_char_value_inn_prod}, we get

        \begin{align*}
           \wtilde{\chi}_\lambda(\sigma_\mu)&=2^{\frac{l(\mu)-l(\lambda)-d(\lambda)}{2}}\sum_{\hat{\lambda}\in I_{\lambda}(\mu_k)}\varepsilon(SD_{\lambda/\hat{\lambda}})\langle Q_{\hat{\lambda}},p_{\hat{\mu}}\rangle =2^{\frac{l(\mu)-l(\lambda)-d(\lambda)}{2}}\sum_{S=(\lambda^{(0)},\ldots,\lambda^{(k)})}\varepsilon(S).
        \end{align*}

        \medskip

        \noindent Finally, we can also state the rule in terms of shifted border strip tableaux. Let $\lambda\in \DP_n$ and $\mu=(\mu_1,\ldots,\mu_k)\in \OP_n$.  A \textit{ shifted border-strip tableaux} (abbrv. SBST) of shape $\lambda$ and type $\mu$ is a filling of the shifted diagram $SD_{\lambda}$ such that  (a) $i$ appears $\mu_i$ times, (b) the entries are weakly increasing in both rows and columns, and (c) the skew diagram $\delta_i$ corresponding to the entries $i$ is either a border strip or a double strip for every $1\leq i\leq k$. Let us denote the set of all shifted border strip tableaux of shape $\lambda$ and weight $\mu$ by $\text{SBST}(\lambda,\mu)$. 
        \begin{equation}\label{Morrris_rule_Sntilde}
            \wtilde{\chi}_\lambda(\sigma_\mu)=2^{\frac{l(\mu)-l(\lambda)-d(\lambda)}{2}}\sum_{T\in \mbox{SBST}(\lambda,\mu)}\varepsilon(T),
        \end{equation}
        where $\displaystyle \varepsilon(T):=\prod_{i=1}^{k}\varepsilon(\delta_i)$. We mention that when $\mu \in \DP_n^1$, then the value of $\wtilde{\chi}_{\lambda}(\sigma_{\mu})$ is completely determined (see \cite[Theorem 8.7]{hoff}).

\section{Irreducible spin Representations of \texorpdfstring{$G\wr \wtilde{S}_n$}{G wr Sn}}\label{irreps_of_GwrSntilde}
    In this section, we discuss the conjugacy classes and complex irreducible spin characters of the group $G\wr \wtilde{S}_n$, where $G$ is any finite group. We closely follow \cite{fren}. Given a positive integer $n$, let $G^n=G\times \cdots \times G$ be the $n$-th direct product of $G$. $\wtilde{S}_n$ acts on $G^n$ via the permutation action of $S_n$, that is,
    \begin{equation}
        t_{\sigma}(g_1,\cdots,g_n)=(g_{\sigma^{-1}(1)},\cdots,g_{\sigma^{-1}(n)}),
    \end{equation}
    where $t_{\sigma}=\theta^{-1}(\sigma)$, and $z$ acts trivially on $G^n$, that is,
    \begin{equation}
        z(g_1,\cdots,g_n)=(g_1.\cdots, g_n).
    \end{equation}
    This yields the spin wreath product $\wtilde{G}_n:=G\wr \wtilde{S}_n$. Clearly, $\wtilde{G}_n$ is a central extension of $G_n:=G\wr S_n$ by $\mathbb{Z}/2\mathbb{Z}$, giving a short exact sequence
    $$\displaystyle 1\rightarrow \mathbb{Z}/2\mathbb{Z}\xrightarrow{i}\wtilde{G}_n \xrightarrow{\wtilde{\theta}} G_n \xrightarrow{i}1,$$
    where $\wtilde{\theta}:\wtilde{G}_n\to G_n$ is defined by $(g_1,g_2,\ldots,g_n,x)\mapsto (g_1,g_2,\ldots,g_n,\theta(x))$. We have $|\wtilde{G}_n|=2n!|G|^n$. Recall that $d$ is the parity homomorphism from the spin group $\wtilde{S}_n$ to the group $\mathbb{Z}_2$ by
    \begin{equation}
        d(t_i)=1\ (i=1,\cdots, n-1),\ d(z)=0.
    \end{equation}
    This parity map extends trivially to $\wtilde{G}_n$ by  $d(g,t_i)=1\ (i=1,\cdots,n-1),\ d(g,z)=0.$

    \subsection{Partition-valued functions and conjugacy classes of \texorpdfstring{$G\wr\wtilde{S}_n$}{G wr Sn}} \label{conj_classes_GwrSntilde}
    
    Let $X$ be any finite set. Let $\p(X)$ be the set of all partition-valued functions on $X$. If $\rho \in \p(X)$, let $\displaystyle ||\rho||=\sum_{c\in X}|\rho(c)|$. Let $\displaystyle l(\rho):=\sum_{x\in X}l(\rho(x))$ denote the length of $\rho$. Let $\p_n(X)$ denote the set of all partition-valued functions $\rho$ on $X$ such that $||\rho||=n$. Let $\OP(X)$ denote the set of all partition-valued functions $(\rho(x))_{x\in X}$ such that $\rho(x)\in \OP$ for every $x\in X$. Let $\DP(X)$ be the set of all  partition-valued functions $(\rho(x))_{x\in X}$ such that $\rho(x)\in \DP$ for every $x\in X$. We set $\OP_n(X):=\p_n(X)\cap \OP(X)$ and $\DP_n(X):=\p_n(X)\cap \DP(X)$.
    
    \medskip
    
    \noindent For a partition-valued function $\rho=(\rho(x))_{x \in X}$ on $X$, let $d(\rho):=||\rho||-l(\rho)\;(\Mod\;2)$ denote the parity of $\rho$. For $i=0,1$, define
    \begin{equation}
        \p^i_n(X)=\{\rho \in \p_n(X)\ |\ d(\rho)=i\},
    \end{equation}
    Then $\rho$ is called even if it is an element of $\p^0_n(X)$, otherwise it is called odd. For $i=0,1$, let 
    \begin{equation}
        \DP^i_n(X)=\p^i_n(X) \cap \DP_n(X).
    \end{equation}
    \begin{lemma}\cite[Lemma 2.3]{fren}
        We have $|\OP_n(X)|=|\DP_n(X)|$.
    \end{lemma}

    In order to parameterize the conjugacy classes of $\wtilde{G}_n$, we first describe the conjugacy classes of $G_n:=G\wr S_n$. Let $\C(G)=\{c_1,c_2,\cdots, c_s\}$ be a labeling of the conjugacy classes of $G$. Let $x=(g_1,\cdots,g_n,\pi)\in G_n$ and $\pi=\pi_1\cdots\pi_k$ be a disjoint cycle decomposition of $\pi$. If $\pi_i=(i_1,i_2,\ldots,i_t)$, then $g_i(x)=g_{i_t}g_{i_{t-1}}\cdots g_{i_1}\in G$ is uniquely determined up to conjugation in $G$. We say that $g_i(x)$ is a cycle product of $g$ corresponding to the cycle $\pi_i$. This allows us to naturally associate $x\in G_n$ with a partition-valued function $\rho$ on $\C(G)$ as follows: $c_j$ maps to a partition $\lambda=\langle 1^{m_1},\cdots,i^{m_i}, \cdots \rangle$ if and only if there exist $m_i$ cycles of length $i$ in the disjoint cycle decomposition of $\pi$ whose corresponding cycle products belong to $c_j$. It is trivially assumed that if there is no cycle in the disjoint cycle decomposition of $\pi$ whose cycle product lies in $c_j$, then $c_j$ maps to the empty partition of $0$. Since $\pi \in S_n$, it is clear that $||\rho||=n$. We call $\rho$ the type of $x$ and denote it by $\text{ty}(x)$. It is then a standard result that if $x,y\in G_n$, then $x$ and $y$ are conjugate in $G_n$ if and only if $\text{ty}(x)=\text{ty}(y)$. Of course, it is easy to see that if $\rho \in \p_n(\C(G))$, then there exists $x\in G_n$ such that $\text{ty}(x)=\rho$. Thus, the conjugacy classes of $G_n$  are parametrized by partition-valued functions $\rho$ on $\C(G)$ such that $||\rho||=n$. For $\rho \in \p_n(\C(G))$, let $\C_{\rho}$ be the conjugacy class of $G_n$ indexed by $\rho$.

    \medskip

    \noindent Let $\rho \in \p_n(\C(G))$. It is easy to see that the pre-image $\wtilde{\theta}^{-1}(\C_{\rho})$ is either a conjugacy class of $\wtilde{G}_n$ (which is the case if $x$ and $zx$ are conjugate in $\wtilde{G}_n$ for every $x\in \wtilde{\theta}^{-1}(\C_{\rho})$), or splits into two conjugacy classes of $\wtilde{G}_n$ (which is the case if $x$ and $zx$ are not conjugate in $\wtilde{G}_n$ for every $x\in \wtilde{\theta}^{-1}(\C_{\rho})$). The following theorem determines the splitting criteria of $\wtilde{\theta}^{-1}(\C_{\rho})$.

    
    \begin{theorem}\cite[Theorem 2.5]{fren}\label{split_classes_GwrSntilde}
        Let $\rho\in \p_n(\C(G))$ and $\C_{\rho}$ denote the conjugacy class of $G_n$ indexed by $\rho$. Then the pre-image $\wtilde{\theta}^{-1}(\C_{\rho})$ splits into two conjugacy classes in $\wtilde{G}_n$ if and only if $\rho \in \OP_n(\C(G))\cup\DP^1_n(\C(G))$.
   \end{theorem}
   
    \noindent Let $\rho \in \p_n(\C(G))$ and set $\D_{\rho}:=\wtilde{\theta}^{-1}(\C_{\rho})$. If $\rho \notin \OP_n(\C(G))\cup \DP_n^1(\C(G))$, then $\D{\rho}$ is a called a non-split conjugacy class of $\wtilde{G}_n$. Otherwise, $\D_{\rho}$ splits into two conjugacy classes in $\wtilde{G}_n$, called split classes. Thus, the number of conjugacy classes of $\tilde{G}_n$ is $|\p_n(\C(G))|+|\OP_n(\C(G))|+|\DP_n^1(\C(G))|$. Recall that for a partition $\lambda \vdash n$, $z_{\lambda}$ is the order of the centralizer of an element $x\in \C_{\lambda}$ in $S_n$. Let $\rho\in \p_n(\C(G))$.  It is well-known that 
    $$z_\rho:=\prod_{c \in \C(G)}z_{\rho(c)}\zeta_c^{l(\rho(c))},$$
    is the order of the centralizer of an element $x\in \C_{\rho}$ in $G_n$, where $\zeta_c$ denotes the order of the centralizer of an element of the conjugacy class $c$ of $G$.

    \begin{proposition}\cite[Proposition 2.6]{fren}\label{centralizer_split_classes_GwrSntilde}
        Let $\rho \in \p_n(\C(G))$. The order of the centralizer of an element $x\in \D_{\rho}$ of $\wtilde{G}_n$ is
        \[\wtilde{z_{\rho}}=\begin{cases} 
            2z_{\rho} & \text{if $\D_{\rho}$ splits, that is, $\rho \in \OP_n(\C(G))\cup \DP_n^{1}(\C(G))$}\\
            z_{\rho} & \text{otherwise}.
            \end{cases}
        \]
    \end{proposition}

    \noindent For $\rho \in \OP_n(\C(G))\cup \DP_n^1(\C(G))$, we label the two split classes as follows: Let $(h_1,h_2,\ldots, h_n,\pi)\in \C_{\rho}$ and assume that $\ty(\pi)=\mu\in \p_n$. Clearly, $\mu$ is a decreasing rearrangement of $\cup_{i}\rho(c_i)$. We write $\mu=(\mu^i_c)_{c\in \C(G)}$, where $\mu^{i}_c$ is a part of $\rho(c)$. There exists an element $x=(g_1,g_2,\ldots,g_n, \sigma_{\mu})\in \wtilde{G}_n$ such that $\wtilde{\theta}(x)\in \C_{\rho}$, that is, $x\in \D_{\rho}$. Here $\sigma_{\mu}\in \wtilde{S}_n$ is as defined in Section \ref{conj_classes_Sntilde}. We label the conjugacy class of $x$ as $\D^+_\rho$ and the conjugacy class of $zx$ as $\D_{\rho}^{-}$. We denote the element $x$ by $\sigma_{\rho}$.

    \subsection{Irreducible spin characters of \texorpdfstring{$\wtilde{G}_n$}{G	extasciitilde{}n}} We now briefly describe the irreducible spin representations (and characters) of $\wtilde{G}_n$. 
    
        \begin{definition}
            An irreducible representation $\varphi$ of $\wtilde{G}_n$ is called spin if $\varphi(z)=-I$. Let $(-1)^d$ be the sign representation of $\wtilde{G}_n:\wtilde{x}\longrightarrow (-1)^{d(\wtilde{x})}$. When $(-1)^d\varphi \simeq \varphi$, we call $\varphi$ a double spin representation of $\wtilde{G}_n$. Otherwise, $\varphi$ and $\varphi^a:=(-1)^d\varphi$ are called a pair of associate spin representations of $\wtilde{G}_n$.
        \end{definition}
    
        \noindent Let $\wtilde{G}_m\wtilde{\times}\wtilde{G}_n$ be the cartesian product of $\wtilde{G}_m$ and $\wtilde{G}_n$ with the following twisted multiplication
        $$(t,t').(s,s')=(tsz^{d(t')d(s)},t's'),$$
        where $s,t\in \wtilde{G}_m$ and $s',t'\in \wtilde{G}_n$. $\wtilde{G}_m\wtilde{\times}\wtilde{G}_n$ is a group and $|\wtilde{G}_m\wtilde{\times}\wtilde{G}_n|=|\wtilde{G}_m||\wtilde{G}_n|$. Define the \textit{spin product} of $\wtilde{G}_m$ and $\wtilde{G}_n$ by
        $$\wtilde{G}_m\hat{\times}\wtilde{G}_n=\wtilde{G}_m\wtilde{\times}\wtilde{G}_n/\{(1,1),(z,z)\},$$
        which can be embedded into the spin wreath group $\wtilde{G}_{m+n}$ canonically by letting
        $$((g,t_i'),1)\mapsto (g,t_i),\ (1,(g,t_j^{''}))\mapsto (g,t_{n+j}),$$
        where $i=1,\cdots,m-1$ and $j=1,\cdots,n-1$. We identify $\wtilde{G}_m\hat{\times}\wtilde{G}_n$ with its image in $\wtilde{G}_{m+n}$ and regard it as a subgroup of $\wtilde{G}_{m+n}$. 

        \medskip
        
        \noindent For any two irreducible spin representation $\varphi$ and $\varphi'$ with representation spaces $U$ and $V$ of $\wtilde{G}_m$ and $\wtilde{G}_n$, the \textit{super (outer) tensor product} $\varphi\widehat{\tensor}\varphi'$ of $\wtilde{G}_m\hat{\times}\wtilde{G}_n$ with representation space $U\otimes V$ is defined by
        $$(t,s).(u\otimes v)=(-1)^{d(s)d(u)}(\varphi(t)(u)\otimes \varphi'(s)(v)),$$
        where $t\in \wtilde{G}_m,\ s\in \wtilde{G}_n$ and $u\in U,\ v\in V$. It is immediate that
        $$(z',z^{''}).(u\otimes v)=(-u)\otimes (-v)=u\otimes v,\;\;(z,1).(u\otimes v)=-(u\otimes v).$$
        The irreducibility of  $\varphi\widehat{\tensor}\varphi'$ is discussed in \cite[Proposition 3.7]{fren}.
        \medskip
        
        \noindent Let $\irr(G)$ denote the set of all complex irreducible characters of $G$ and we fix a labelling $\{\gamma_{1},\cdots, \gamma_{s}\}$ of $\irr(G)$. Recall that $\wtilde{\chi}_{\lambda}$ is the irreducible spin  character of $\wtilde{S}_n$ indexed by the strict partition $\lambda\vdash n$. Let $\varphi$ be an irreducible representation of $G$ with representation space $V$ and character $\gamma$.  If the irreducible character $\wtilde{\chi}_{\lambda}$ of $\wtilde{S}_n$ is afforded by the representation $\psi_{\lambda}$ with representation space $\wtilde{V}_{\lambda}$, then a spin  representation of $\wtilde{G}_n$ can be obtained with representation space $T^n(V):=\underbrace{V\tensor V\tensor \cdots \tensor V}_{n\;\text{times}}\tensor \wtilde{V}_{\lambda}$ with the following action: If $x=(g_1,\ldots,g_n,t_{\pi})\in \wtilde{G}_n$ such that $\wtilde{\theta}(x)=(g_1,g_2,\ldots, g_n,\pi)\in G_n$ , then $x.(v_1\tensor v_2 \tensor\cdots\tensor v_n \tensor u):=\gamma(g_1)(v_{\pi^{-1}(1)})\tensor \cdots \tensor \gamma(g_n)(v_{\pi^{-1}(n)})\tensor \psi_{\lambda}(t_{\pi})(u)$. This is an irreducible spin representation of $\wtilde{G}_n$. We denote it by $\widetilde{\tensor^n}\varphi\tensor \psi_{\lambda}$, and its character by $\wtilde{\chi}_\lambda(\gamma)$, whose values are given in \cite[Proposition 4.4]{fren} . Recall that $\wtilde{\chi}_n$ is the basic spin character of $\wtilde{S}_n$. Thus, if $\lambda=(n)$,  we call $\wtilde{\chi}_n(\gamma)$ a basic spin character of $\wtilde{G}_n$ corresponding to the irreducible character $\gamma$ of $G$. In particular, if $\gamma$ is trivial, we note that $\chi_n(\gamma)(\sigma_{\rho})$ is positive when $\rho \in \OP_n(\C(G))$ (see \cite[Corollary 4.5]{fren}).

        \medskip
        
        \noindent More generally, if $\rho \in \DP_n(\irr(G))$ (with the fixed labelling of $\irr(G)$), and $\rho(\gamma_i)=\lambda^i\vdash n_i$ for $1\leq i\leq s$, then an irreducible spin representation (or spin character) of $\wtilde{G}_n$ can be constructed the following way: for $1\leq i\leq s$, let $\varphi_i$ be the irreducible representation of $G$ that affords $\gamma_i$ and let $\psi_{\lambda^i}$ be the irreducible spin representation of $\wtilde{S_{n_i}}$ indexed by the strict partition $\lambda_i$. Then, $\widetilde{\tensor^{n_i}} \varphi_i\tensor \psi_{\lambda^i}$ is an irreducible spin representation of $\wtilde{G}_{n_i}$. The super tensor product of these spin representations, that is, $\overset{i}{\widehat{\tensor}}(\widetilde{\tensor^n} \varphi_i \tensor \psi_{\lambda^i})$ yields a spin representation of $\hat{\times}_i\wtilde{G}_{n_i}:=\wtilde{G}_{n_1}\hat{\times} \cdots \hat{\times} \wtilde{G}_{n_s}$. This splits into $[m/2]$ copies of a single irreducible spin representation of $\hat{\times}_i\wtilde{G}_{n_i}$, where $m=|\{i\mid d(\lambda^i)=1\}|$ (see \cite[Proposition 9.8]{fren}). We denote this representation by $\Phi_{\rho}$. Finally, the induction of $\Phi_{\rho}$ to $\wtilde{G}_n$ (recall that $\displaystyle \sum_{i=1}^{s}n_i=n$), that is, $\psi_{\rho}:=\ind_{\hat{\times}_{i}\wtilde{G}_{n_i}}^{\wtilde{G}_n} \Phi_{\rho}$ yields a representation $\wtilde{G}_n$. These turn out to be inequivalent irreducible representations (in other words, corresponding irreducible characters are distinct) of $\wtilde{G}_n$ corresponding to distinct strict partition-valued functions $\rho$. Thus, we get distinct irreducible spin representations of $\wtilde{G}_n$ that are indexed by elements of  $\DP_n(\irr(G))$ (with the fixed labeling of $\irr(G)$). Given $\rho\in \DP_n(\irr(G))$, let $\wtilde{\chi}_{\rho}$ be the character of $\psi_{\rho}=\ind_{\hat{\times}_{i}\wtilde{G}_{n_i}}^{\wtilde{G}_n} \Phi_{\rho}$ indexed by $\rho$. We have the following theorem.
        
        \begin{theorem}\cite[Theorem 9.7]{fren}
            There exists an irreducible double spin representation $\psi_{\rho}$ of $\wtilde{G}_n$ corresponding to each partition-valued function $\rho \in \DP_n^{0}(\irr(G))$, while there is a pair of irreducible associate spin representations $(\psi_{\rho}, \psi_{\rho}^a)$ of $\wtilde{G}_n$ corresponding to each $\rho \in \DP_n^1(\irr(G))$. Thus, there is a total of $|\DP_n(\irr(G))|+|\DP_n^1(\irr(G))|$ irreducible spin representations of $\wtilde{G}_n$.
         \end{theorem}

\section{The Frobenius characteristic on the space of spin class functions of \texorpdfstring{$\wtilde{G}_n$}{G wr Sn}}\label{Frob_char_GwrSntilde}
    As mentioned in the introduction, the authors in \cite{fren} developed a symmetric function theory associated with the irreducible spin characters of $\wtilde{G}_n$. We discuss it here. Eventually, we will use this approach to prove our main theorem. We use notations in accordance with \cite[Appendix B]{mac}, where the author develops a symmetric function approach to classify all complex irreducible characters of $G_n$.
    \subsection{The algebra \texorpdfstring{$R^-({\wtilde{G}})$}{R power G tilde}}
        Let us denote the space of all class functions of $\wtilde{G}_n$ over $\mathbb{C}$ by $R(\wtilde{G}_n)$ and $\displaystyle R(\wtilde{G}):=\oplus_{n\geq 0}R(\wtilde{G}_n)$. 
        Let $R^+(\wtilde{G}_n)$ be the subspace of all class functions $f$ such that $f(zx)=f(x),\ x\in \wtilde{G}_n$. A \textit{spin class function} on $\wtilde{G}_n$ is a class function from $\wtilde{G}_n$ to $\mathbb{C}$
        such that $f(zx)=-f(x).$ It is clear that spin class functions vanish on non-split conjugacy classes. Let $R^-(\wtilde{G}_n)$ be the  space of spin class functions of $\wtilde{G}_n$ over $\mathbb{C}$ and $R^-(\wtilde{G})=\oplus_{n \geq 0}R^-(\wtilde{G}_n).$
        For $\rho \in \OP_n(\C(G))\cup \DP^1_n(\C(G))$, define
        \[{1'_{\rho}}(x)=\begin{cases}
            1 & \text{if $\ x \in \D_{\rho}^+$},\\
            -1 & \text{if $\ x\in \D_{\rho}^-$},\\
            0 & \text{otherwise}.
            \end{cases}
        \]
        Thus,
        \[R^-(\wtilde{G}_n)=\text{$\mathbb{C}$-span of $\{1'_{\D_\rho} \mid \ \rho \in \OP_n(\C(G))\cup \DP^1_n(\C(G))\}$},\]
        and
        \[R(\wtilde{G}_n)=R^+(\wtilde{G}_n)\oplus R^-(\wtilde{G}_n);\;R^+(\wtilde{G}_n)\simeq R(G_n).\]
        A \textit{spin super class function} of $\wtilde{G}_n$ is a spin class function $f$  such that $f$ vanishes on $\D_{\rho}$, for every $\rho \in \DP_n^{1}(\C(G))$. Let $R^0(\wtilde{G}_n)$ be set of all spin super class functions on $\wtilde{G}_n$ and set $R^0(\wtilde{G}):=\oplus_{n \geq 0}R^0(\wtilde{G}_n).$
        Clearly, $R^0(\wtilde{G}_n)=\text{$\mathbb{C}$-span of $\{1'_{\rho} \mid  \rho \in \OP_n(\C(G))\}$}\subseteq R^-(\wtilde{G}_n).$

        \medskip
        
        \noindent The standard inner product $\langle.,.\rangle$ on $R(\wtilde{G}_n)$ induces an inner product on $R^{0}(\wtilde{G}_n)$. For $f,g\in R^{0}(\wtilde{G}_n)$, we get
        $$\langle f,g \rangle_{\wtilde{G}_n} =\frac{1}{|\wtilde{G}_n|}\sum_{x\in \wtilde{G}_n^0}f(x)g(x^{-1})=\frac{1}{2}\langle f,g \rangle_{\wtilde{G}_n^{0}},$$
        where $\langle f,g \rangle_{\wtilde{G}_n^{0}}$ is the inner product of $f|_{\wtilde{G}_n^0}$ and $g|_{\wtilde{G}_n^0}$ in the space of class functions of the subgroup $\wtilde{G}_n^0=G \wr \wtilde{A}_n$. The above inner product can be written as 
        \begin{equation}\label{equ.a}
            \langle f,g \rangle =\sum_{\rho \in\OP_n(\C(G))}\frac{1}{z_\rho}f(\sigma_\rho)\overline{g(\sigma_\rho)}.
        \end{equation}
        
    \subsection{The algebra \texorpdfstring{$\Delta(G)$}{Delta G}}
        Let $p_r(c)\ (r\geq 1,\ c\in \C({G}))$ be independent indeterminates over $\mathbb{C}$ with degree $r$. Here we may think of $p_r(c)$ as the $r$-th power sum symmetric functions in the variables $(x_{ic})_{i\geq 1}$. Let
        $$\Lambda({G}):=\mathbb{C}[p_r(c)\ |\ r\geq 1,\ c\in (\C({G})].$$
        If $\lambda=(\lambda_1,\lambda_2,\cdots)$ is a partition and $c\in \C(G)$, define $p_\lambda(c):=p_{\lambda_1}(c)p_{\lambda_2}(c)\cdots$. Now if $\rho\in \p(\C(G))$, then define
        $$P_\rho:=\prod_{c\in \C({G})} p_{\rho(c)}(c).$$
        Clearly, $\{P_\rho\ |\ \rho \in \p(\C({G}))\}$ is a basis for $\Lambda({G})$.
        Let $\Lambda_n({G})$ be the homogeneous component of $\Lambda({G})$ of degree $n$. We have
        \[\Lambda_n({G})=\langle P_\rho \mid \rho\in \p_n(\C(G))\rangle ,\;\text{and}\;\Lambda({G})=\oplus_{n\geq 0}\Lambda_n({G}).\]

        \medskip

        Define the sub-algebra $\Delta({G})$ of $\Lambda({G})$ by
        $$\Delta({G}):=\mathbb{C}[p_r(c)\ |\ r \ \mbox{is a odd positive integer and}\ c \in \C(G)].$$
        Clearly, $$\Delta({G})=\langle P_\rho\ |\ \rho \in \OP(\C(G))\rangle.$$
        Let $\Delta_n({G})$ be the homogeneous component of $\Delta({G})$ of degree $n$. Then
        $$\Delta_n({G})=\langle P_\rho\ |\ \rho\in \OP_n(\C(G))\rangle,\;\;\text{and}\quad\Delta({G})=\oplus_{n\geq 0}\Delta_n({G}).$$
        For $\rho$ and $\rho'\in \OP_n(\C(G))$ define an inner product on $\Delta_n({G})$ by
        $$\langle P_\rho, P_{\rho'} \rangle:=z_\rho2^{-l(\rho)}\delta_{\rho \rho'},$$
        where $z_\rho$ is the size of the centralizer of an element of $G\wr S_n$ whose conjugacy class is indexed by $\rho$.

    \subsection{The odd characteristic och} The odd characteristic map is the linear transformation $\och_n: R^-(\wtilde{G}_n)\rightarrow \Delta_n({G})$ defined by
    \[\och_n(1'_{\rho})=\begin{cases}
    \frac{2^{\frac{l(\rho)}{2}}}{z_\rho}P_\rho & \text{if}\ \rho \in \OP_n(\C(G)),\\
    0 & \text{if}\ \rho \in \DP^1_n(\C(G)).\end{cases}\]
    Thus, for $f \in R^-({\wtilde{G}_n})$, we have
    $$\och_n(f)=\sum_{\rho \in \OP_n(\C(G))}\frac{2^{\frac{l(\rho)}{2}}}{z_\rho}f(\sigma_{\rho})P_\rho.$$  We remark that it is convenient to define $\och_n$ on $R^{-}(\wtilde{G}_n)$ in place of $R^{0}(\wtilde{G}_n)$ (which is done in \cite[6.1]{fren}). We linearly extend the map $\och_n$ from $R^-({\wtilde{G}})$ to $\Delta({G})$, which we call \textbf{och}. We include a proof of the next result for clarity.
    
    \begin{theorem}\cite[Proposition 6.4]{fren}\label{isometry}
        The map $\och_n|{_{R^0(\wtilde{G}_n)}}:R^0(\wtilde{G}_n)\rightarrow \Delta_n({G})$ is an isometry.
    \end{theorem}
    
    \begin{proof}
        Let $\rho,\rho' \in \OP_n(\C(G))$. Recall that $\langle P_\rho, P_{\rho'} \rangle=z_\rho2^{-l(\rho)}\delta_{\rho \rho'}.$
        Let $f, g \in R^0({\wtilde{G}_n})$.  Then, 
        \begin{align*}
        \langle \och_n(f),\och_n(g) \rangle&=\langle\sum_{\rho \in \OP_n(\C(G))}\frac{2^{\frac{l(\rho)}{2}}}{z_\rho}f(\sigma_{\rho})P_\rho \ ,\sum_{\rho' \in \OP_n(\C(G))}\frac{2^{\frac{l(\rho')}{2}}}{z_{\rho'}}g(\sigma_{\rho'})P_{\rho'} \rangle\\
        &=\sum_{\rho, \rho' \in \OP_n(\C(G))}\frac{2^{\frac{l(\rho)+l(\rho')}{2}}}{z_\rho z_{\rho'}}f(\sigma_{\rho})\overline{g(\sigma_{\rho'})} \langle P_\rho, P_{\rho'}\rangle\\
        &=\sum_{\rho, \rho' \in \OP_n(\C(G))}\frac{2^{\frac{l(\rho)+l(\rho')}{2}}}{z_\rho z_{\rho'}}f(\sigma_{\rho})\overline{g(\sigma_{\rho'})} z_\rho 2^{-l(\rho))}\delta_{\rho,\rho'}\\
        &=\sum_{\rho \in \OP_n(\C(G))}\frac{2^{l(\rho)}}{z_\rho z_{\rho}}f(\sigma_{\rho})\overline{g(\sigma_{\rho'})} z_\rho 2^{-l(\rho))}
        =\sum_{\rho \in \OP_n(\C(G))}\frac{1}{z_\rho}f(\sigma_{\rho})\overline{g(\sigma_{\rho'})}=\langle f, g \rangle.
        \end{align*}
        
    \end{proof}

    \begin{corollary}\label{cor.isometry}
        Let $f\in R^-({\wtilde{G}_n})$ and $g\in R^0({\wtilde{G}_n})$, then $\langle \och_n(f),\och_n(g) \rangle=\langle f, g \rangle.$
    \end{corollary}

    \begin{proof}
        Suppose, $f\in R^-({\wtilde{G}_n})$ and $g\in R^0({\wtilde{G}_n})$. Then we can write $f=f_1+f_2$, where $\displaystyle f_1= \sum_{\rho \in \OP_n(\C(G))}\frac{2^{\frac{l(\rho)}{2}}}{z_\rho}f(\sigma_{\rho})P_\rho$ and $\displaystyle f_2=\sum_{\rho' \in \DP^1_n(\C(G))}\frac{2^{\frac{l(\rho')}{2}}}{z_{\rho'}}f(\sigma_{\rho'})P_{\rho'}$. Since for $\rho \in \OP_n(\C(G))\cup \DP_n^1(\C(G))$, $P_\rho$'s forms an orthogonal basis, so $\langle f,g \rangle=\langle f_1+f_2 ,g \rangle=\langle f_1,g\rangle$. Now, 
        \begin{align*}
            \langle \och_n(f),\och_n(g) \rangle&=\langle \och_n(f_1+f_2),\och_n(g) \rangle=\langle \och_n(f_1)+\och_n(f_2),\och_n(g) \rangle\\
            &=\langle \och_n(f_1),\och_n(g)\rangle+\langle\och_n(f_2),\och_n(g) \rangle\\
            &=\langle \och_n(f_1),\och_n(g)\rangle=\langle f_1, g \rangle=\langle f, g \rangle
        \end{align*}
        for all $f\in R^-({\wtilde{G}_n})$ and $g\in R^0({\wtilde{G}_n})$.   
    \end{proof}
    
    \subsection{Image of an irreducible  spin character of \texorpdfstring{$\wtilde{G}_n$}{GwrSntilde} under the och map}
        
        Recall that $\irr(G)$ is the set of all irreducible characters of $G$ with a fixed labeling $\irr(G)=\{\gamma_1,\cdots, \gamma_s\}$. For each irreducible character $\gamma \in \irr(G)$ and for $r\geq 1$, define
        $$p_r(\gamma)=\sum_{c\in \C(G)} \frac{\gamma(c)p_r(c)}{z_c},$$
        where $z_c$ is the size of the centralizer of an element $x\in c$.\par
        \noindent By the dual orthogonality relation 
        $$p_r(c)=\sum_{\gamma \in \irr(G)}\gamma({c^{-1}})p_r(\gamma).$$
        Thus, $\{p_r(\gamma)\mid \gamma\in \irr(G)\}$ is algebraically independent and generates $\Lambda({G})$ as a $\mathbb{C}$-algebra. Moreover, $\{p_r(\gamma)\mid r\ \text{odd}, \gamma\in \irr(G)\}$ generates $\Delta(G)$ as $\mathbb{C}$-algebra. If $\lambda=(\lambda_1,\lambda_2,\cdots)$ is a partition and $\gamma \in \irr(G)$, define $p_\lambda(\gamma):=p_{\lambda_1}(\gamma)p_{\lambda_2}(\gamma)\cdots$.

        \medskip
        
        \noindent For $\gamma \in \irr(G)$, define
        $$q_n(\gamma):=\sum_{\lambda \in \OP_n}2^{l(\lambda)}p_\lambda(\gamma)/z_\lambda.$$
        
    
        \noindent Let $\displaystyle \lambda \in \DP_n$. Then $\displaystyle Q_\lambda=\sum_{\mu \in \OP_n}{a_{\lambda \mu}} p_\mu,$ where $a_{\lambda \mu}=\frac{\langle Q_\lambda, p_\mu\rangle}{\langle p_\mu,p_\mu \rangle}=\frac{2^{l(\mu)}\langle Q_\lambda, p_\mu\rangle}{z_{\mu}}$. For $\gamma \in \irr(G)$, define
        $$Q_\lambda(\gamma):=\sum_{\mu \in \OP_n}{a_{\lambda \mu}} p_\mu(\gamma).$$
        Thus, if $\lambda=(n)$, we obtain $Q_{(n)}(\gamma)=q_n(\gamma)$. Let $\rho \in \DP_n(\irr(G))$. Define
        $$Q_\rho:=\prod_{\gamma\in \irr(G)}Q_{\rho(\gamma)}(\gamma).$$  

        \begin{theorem}\cite[Proposition 9.5]{fren}\label{prthogonality_Qrho}
            Let $\rho, \rho'\in \DP_n(\irr(G))$. Then $\langle Q_\rho, Q_{\rho'}\rangle=2^{l(\rho)}\delta_{\rho \rho'}.$ In particular, $\{Q_{\rho}\mid \rho \in \DP_n(\irr(G))\}$ forms an orthogonal basis of $\Delta_n(G)$.
        \end{theorem}
       
       \noindent Set ${Q_\rho}^*:=2^{-l(\rho)/2} Q_\rho$.  \noindent The following theorem determines the image of any irreducible spin character of $\wtilde{G}_n$.
        
        \begin{theorem}\cite[Theorem 9.7]{fren}\label{image_spin_char_Gntilde}
            Let $\rho \in \DP_n(\irr(G))$. We have $\och_n(\wtilde{\chi}_\rho)=2^{-\frac{d(\rho)}{2}}Q_\rho^*.$
        \end{theorem}
        

\section{A Morris recursion formula for \texorpdfstring{$\wtilde{G}_n$}{G wr Sn}}\label{sec_main_theorem}
    In this section, we prove our main theorem. We begin with a lemma.
     \begin{lemma}\label{spin_char_value_inn_prod_Gntilde}
        Let $\rho \in \DP_n(\irr(G))$ and $\rho' \in \OP_n(\C(G))$. Then $\wtilde{\chi}_{\rho}(\sigma_{\rho'})=2^{\frac{l(\rho')-l(\rho)-d(\rho)}{2}}\langle {Q_\rho}, P_{\rho'}\rangle$.  
    \end{lemma}
        
    \begin{proof}
        We have
        \begin{align*}
            \langle \wtilde{\chi}_\rho, 1'_{\rho'} \rangle =\frac{1}{|\wtilde{G}_n|}\left[\wtilde{\chi}_\rho(\D^+_{\rho'})|\D^+_{\rho'}|-\wtilde{\chi}_\rho(\D^-_{\rho'})|\D^-_{\rho'}|\right]
            =2\frac{|\D^+_{\rho'}|}{|\wtilde{G}_n|}\wtilde{\chi}_\rho(\D^+_{\rho'})=\frac{\wtilde{\chi}_\rho(\D^+_{\rho'})}{z_{\rho'}}
        \end{align*}
        Now, by Corollary~\ref{cor.isometry}, $\displaystyle {\langle \wtilde{\chi}_\rho, 1'_{\rho'} \rangle}_{R(\wtilde{G}_n)}={\langle \och_n(\wtilde{\chi}_\rho), \och_n(1'_{\rho'}) \rangle}_{\Delta_n({G})}=\frac{\wtilde{\chi}_\rho(\D^+_{\rho'})}{z_{\rho'}}$.\\ Therefore, $\displaystyle\langle 2^{-\frac{d(\rho)}{2}}Q_\rho^*,\frac{2^{l(\rho')/2}}{z_{\rho'}}P_{\rho'}\rangle=\frac{\wtilde{\chi}_\rho(\D^+_{\rho'})}{z_{\rho'}}$, whence $\wtilde{\chi}_\rho(\D^+_{\rho'})=2^{\frac{l(\rho')-l(\rho)-d(\rho)}{2}}\langle {Q_\rho}, P_{\rho'}\rangle.$
    \end{proof}

 
    To state our desired recursion rule, we need to introduce some additional notation. Let $\rho\in \DP_n(\C(G))$. We can view $\rho$ as a $s$-partite strict partition, that is, $\rho=(\rho^1, \cdots, \rho^{s})$, where $\rho^i=\rho(c_i)\in \DP$. Let  $SD_{\rho}=(SD_{\rho^1}, \cdots, SD_{\rho^{s}})$ be the shifted diagram of $\rho$. A shifted border-strip tableaux $T$ of shape $\rho$ and type $\mu=(\mu_1,\cdots,\mu_t)\in \OP_n$ is defined to be a filling of $SD_{\rho}$ such that (a) $i$ appears $\mu_i$ times and all these appear in exactly one $SD_{\rho^j}$, (b) the entries are weakly increasing in both rows and columns, and (c) the shifted skew-diagram $\delta_i$ corresponding to the entries $i$ is either a border-strip or a double strip for every $1\leq i\leq t$. Let us denote the set of all shifted border strip tableaux of shape $\rho$ and type $\mu$ by $SBST(\rho,\mu)$. Once again, a shifted border strip tableaux of shape $\rho$ and type $\mu$ is equivalent to successive removal of border strips or double strips of size $\mu_t$ to  $\mu_1$ from $SD_{\rho}$ (that is, for each $1\leq i\leq t$, a border strip or double strip of size $\mu_i$ is removed from some $SD_{\rho^j}$) until we reach the empty $s$-partite partition. Let $\varepsilon_T(i):=\varepsilon_T(\delta_i)$ take the value $(-1)^{\h(\delta_i)}$ or $2(-1)^{\h(\delta_i)}$ according as $\delta_i$ is a border strip or a double strip. For a shifted border-strip tableaux $T$ of shape $\rho$ (an $s$-partite strict partition) and type $\mu$, we define $f_T(i)$ (called the index of $i$ in $T$) to be $j\in \{1,2,\ldots,s\}$ where $i$ appears in $SD_{\rho^j}$ (note that $i$ appears in an unique $SD_{\rho^j}$ by (a)). The following example illustrates all these notions.

    \begin{example}
        For $\rho=((6,3,1),(5,2,1), (4,2,1))$ and $\mu=(5,3,3,3,3,3,3,1,1)$, the following is a shifted border-strip tableaux of shape $\rho$ and type $\mu$:
        $$\displaystyle T=(\hspace{.2cm}\ytableausetup{smalltableaux}\begin{ytableau}  2 & 2 & 3 & 3 & 3 & 8 \\
        \none & 2 & 4 & 4 \\
        \none & \none & 4 \end{ytableau}\;,\;\begin{ytableau} 1 & 1 & 1 & 1 & 1  \\
         \none & 5 & 5 \\
        \none & \none & 5
        \end{ytableau}\;,\;\begin{ytableau} 6 & 6 & 6 & 9 \\
        \none & 7 & 7 \\
        \none & \none & 7
        \end{ytableau}\;).$$
        Since $\rho^{(1)}$ is filled by $2,\ 3,\ 4$ and $8$, so $f_T(2)=f_T(3)=f_T(4)=f_T(8)=1$. For the same reason, $f_T(1)=f_T(5)=2$, $f_T(6)=f_T(7)=f_T(9)=3$. Note that $\delta_i$ is a border strip if $i\in \{1,3,4,6,8,9\}$. Thus, $\varepsilon_T(1)=\varepsilon_T(3)=\varepsilon_T(6)=\varepsilon_T(8)=\varepsilon_T(9)=0$ and $\varepsilon_T(4)=-1$. Further, $\delta_i$ is a double strip if $i\in \{2,5,7\}$. In this case, we get $ \varepsilon_T(2)=\varepsilon_T(5)=\varepsilon_T(7)=-2$. 
    \end{example}
    

    
    We are now ready to state and prove the Morris recursion rule for computing values of the irreducible spin characters of $\wtilde{G}_n$.
    
    \begin{theorem}\label{main_theorem}
        Let $G$ be a finite group, $\irr(G)=\{\gamma_1,\cdots,\gamma_s\}$ be a fixed labelling of the irreducible characters of $G$, and $\C(G)=\{c_1,\ldots, c_s\}$ be a fixed labeling of the conjugacy classes of $G$. Let $\rho$ be a strict s-partite partition of $n$ and $\wtilde{\chi}_{\rho}$ denote the irreducible spin character of $\wtilde{G}_n$ indexed by $\rho$ (with respect to the labeling of $\irr(G)$). Let $\rho'\in \OP_n(\C(G))$ and $\sigma_{\rho'}=(g_1,\cdots,g_n,\sigma_{\mu})\in \D_{\rho'}^{+}$, where $\mu\in \OP_n$ is a decreasing rearrangement of $\cup_i\rho'(c_i)$.    Then,

        $$\wtilde{\chi}_\rho(\sigma_{\rho'})=2^{\frac{l(\rho')-l(\rho)-d(\rho)}{2}}\sum_{T\in \mathrm{SBST}(\rho,\mu)}\prod_{i=1}^{t}\varepsilon_T(i)\gamma_{f_T(i)}(g_i(\wtilde{\theta}(\sigma_{\rho'}))),$$
        where $\wtilde{\theta}(\sigma_{\rho'})=(g_1,\cdots,g_n,\theta(\sigma_{\mu}))\in G_n$, $\theta(\sigma_{\mu})=\pi_1\cdots\pi_t$ with $l(\pi_i)=\mu_i$, and $g_i(\wtilde{\theta}(\sigma_{\rho'}))$ denote the cycle product of $\wtilde{\theta}(\sigma_{\rho'})$ with respect to the cycle $\pi_i$.
    \end{theorem}
    
    \begin{proof}
        For $\rho'\in \OP_n(\C(G))$, recall that $\D_{\rho'}^{+}$ is one of the split classes indexed by $\rho'$, represented by $\sigma_{\rho'}=(g_1,\ldots,g_n,\sigma_{\mu})$ (see discussion after Proposition~\ref{centralizer_split_classes_GwrSntilde}), where $\mu=(\mu_1,\mu_2,\ldots,\mu_t)\in \OP_n$ is a decreasing rearrangement of $\cup_{i=1}^{s} \rho'(c_i)$. Let $\theta(\sigma_{\mu})=\pi_1\pi_2\ldots \pi_t$, where $l(\pi_i)=\mu_i$. If the cycle product $g_i(\wtilde{\theta}(\sigma_{\rho'}))$ of $\wtilde{\theta}(\sigma_{\rho'})=(g_1,g_2,\ldots, g_n, \theta(\sigma_{\mu}))$ with respect to the cycle $\pi_i$ belongs to the conjugacy class $c\in \C(G)$, then by definition  of $\D_{\rho'}$, $\mu_i$ appears as a part of $\rho'(c)$. We write $\mu_i(c)$ instead of $\mu_i$ to emphasize the fact that $\mu_i$ appears as a part of  $\rho'(c)$. Thus, we can write $\mu=(\mu_1(c_{j_1}), \mu_2(c_{j_2}), \ldots, \mu_t(c _{j_t}))$, where $c_{j_k}\in \C(G)$ for all $1\leq k \leq s$.

        \medskip

        Let $\nu = (\nu_1,\nu_2, \ldots, \nu_s)$ be an $s$-partite partition of $n$ such that $\nu_i\in \OP$ for all $1\leq i\leq s$. Define $\displaystyle \wtilde{P}_{\nu}:=\prod_{i=1}^{s}p_{\nu_i}(\gamma_i)$. For a partition $\mu\in \OP_n$ and $\gamma\in \irr(G)$, $\displaystyle p_{\mu}(\gamma)=\sum_{\lambda\in \DP_n}2^{-l(\lambda})\langle p_{\mu}, Q_{\lambda}\rangle Q_{\lambda}(\gamma)$.
        Thus, 
        $$\wtilde{P}_{\nu}=\prod_{i=1}^{s}\sum_{\lambda\in \DP_{n_i}}2^{-l(\lambda)}\langle p_{\nu_i}, Q_{\lambda}\rangle Q_{\lambda}(\gamma_i),$$
        where $\nu_i\in \OP_{n_i}$ (so that $\sum_{i=1}^{s}n_i=n$). Thus, 
        $$\wtilde{P}_{\nu}=\sum_{\rho\in \DP_n(\irr(G))}[2^{-l(\rho)}\prod_{i=1}^{s}\langle p_{\nu_i},Q_{\rho(\gamma_i)}\rangle]Q_{\rho}.$$

        \noindent Thus, if $\rho=(\rho^1,\ldots, \rho^s)\in \DP_n(\irr(G))$ is a $s$-partite strict partition of $n$, then the coefficient of $Q_{\rho}$ in the expansion of $\wtilde{P}_{\nu}$ with respect to the basis $\{Q_{\rho}\mid \rho \in \DP_n(\irr(G))\}$ of $\Delta_n(G)$ is given by 
        $$2^{-l(\rho)}\sum_{(T_1,T_2,\ldots, T_s)\in \mathscr{T}} \prod_{i=1}^{s}\varepsilon(T_i),$$
        where $\mathscr{T}:=\text{SBST}(\rho^1,\nu_1)\times \text{SBST}(\rho^2,\nu_2) \times \cdots \times \mbox{SBST}(\rho^s,\nu_s)$, and $\varepsilon(T_i)$ is as defined in Section 2.4. Here it is clear that the coefficient of $Q_{\rho}$ is non-zero implies that $|\rho^i|=|\nu_i|$ for every $1\leq i\leq s$.

        \medskip
        
        \noindent Let $\displaystyle P_{\rho'}=\sum_{\rho \in \DP_n(\irr(G))} L_{\rho'\rho}Q_{\rho}$. We will find an expression for $L_{\rho'\rho}$. Recall that 
        $$P_{\rho'}=\prod_{c\in \C(G)}p_{\rho'(c)}(c)=\prod_{i=1}^{t}p_{\mu_i}(c_{j_i}).$$
        In the above, $\mu_i(c_{j_i})$ is replaced by $\mu_i$ for convenience. Now, $\displaystyle p_{\mu_i}(c_{j_i})=\sum_{\gamma\in \irr(G)}\gamma(c_{j_i}^{-1})p_{\mu_i}(\gamma)$. We obtain the following expansion of $P_{\rho'}$:
        \begin{align*}
            P_{\rho'}&=p_{\mu_1}(c_{j_1})p_{\mu_2}(c_{j_2})\cdots p_{\mu_t}(c_{j_t})\\
            &=\left(\sum_{\gamma \in \irr(G)}\gamma({c_{j_1}^{-1}})p_{\mu_1}(\gamma)\right)\left(\sum_{\gamma \in \irr(G)}\gamma({c_{j_2}^{-1}})p_{\mu_2}(\gamma)\right)\cdots \left(\sum_{\gamma \in \irr(G)}\gamma({c_{j_t}^{-1}})p_{\mu_t}(\gamma)\right)\\
            &=\sum_{f \in \mathscr{F}} \prod_{i=1}^t \gamma_{f(i)}({c_{j_i}^{-1}})p_{\mu_i}(\gamma_{f(i)}),
        \end{align*}
        where $\mathscr{F}$ is the set of all functions from $\{1,2,\cdots, t\} \to \{1,2,\cdots, s\}$.
        We note that $L_{\rho'\rho}$ is the sum of all the coefficients of $Q_{\rho}$ that appear in the expansion of each of the summands $\displaystyle \prod_{i=1}^t \gamma_{f(i)}({c_{j_i}^{-1}})p_{\mu_i}(\gamma_{f(i)})$ for each $f\in \mathscr{F}$, with respect to the basis $\{Q_{\rho}\mid \rho \in \DP_n(\irr(G))\}$ of $\Delta_n(G)$.

        \medskip
        
        \noindent For $f\in \mathscr{F}$, we examine the coefficient of $Q_{\rho}$ that appears in  $\displaystyle \prod_{i=1}^t \gamma_{f(i)}({c_{j_i}^{-1}})p_{\mu_i}(\gamma_{f(i)})$. Consider the $s$-partite partition $\hat{\mu}=(\hat{\mu}^1,\ldots, \hat{\mu}^s)$ where $\hat{\mu}^k$ consists of the parts $\{\mu_{j}\mid f(j)=k\}$ arranged in decreasing order. Thus, the coefficient of $Q_{\rho}$ is given by 
        $$2^{-l(\rho)}\left(\prod_{i=1}^{t}\gamma_{f(i)}(c_{j_i}^{-1})\right)\sum_{(T_1,T_2,\ldots, T_s)\in \mathscr{T}} \prod_{i=1}^{s}\varepsilon(T_i),$$
        where $\mathscr{T}=\text{SBST}(\rho^1,\hat{\mu}^1)\times \text{SBST}(\rho^2,\hat{\mu}^2) \times \cdots \times \text{SBST}(\rho^s,\hat{\mu}^s)$.  Observe that $(T_1,T_2,\ldots, T_s)$ corresponds to a shifted border strip tableaux $T$ of shape $\rho$ and type $\mu$. Moreover, using the observation made in the first paragraph of the proof, we get that the coefficient $\displaystyle \left(\prod_{i=1}^{t}\gamma_{f(i)}(c_{j_i}^{-1})\right)\prod_{i=1}^{s}\varepsilon(T_i)$ is precisely $\displaystyle \prod_{i=1}^{t}\varepsilon_T(i)\overline{\gamma_{f_T(i)}(g_i(\wtilde{\theta}(\sigma_{\rho'})))}$.

        \medskip

        \noindent Thus, we conclude that the coefficient $L_{\rho'\rho}$ is 
        $$2^{-l(\rho)}\sum_{T\in \mathrm{SBST}(\rho,\mu)}\prod_{i=1}^{t}\varepsilon_T(i)\overline{\gamma_{f_T(i)}(g_i(\wtilde{\theta}(\sigma_{\rho'})))}.$$
        Since $\{Q_{\rho}\mid \rho \in \DP_n(\irr(G))\}$ form an orthogonal basis for $\Delta_n(G)$ with $\langle Q_{\rho}, Q_{\rho}\rangle=2^{l(\rho)}$, we conclude that $L_{\rho'\rho}=2^{-l(\rho)}\langle P_{\rho'}, Q_{\rho}\rangle$, whence
        $$\langle P_{\rho'}, Q_{\rho}\rangle=\sum_{T\in \mathrm{SBST}(\rho,\mu)}\prod_{i=1}^{t}\varepsilon_T(i)\overline{\gamma_{f_T(i)}(g_i(\wtilde{\theta}(\sigma_{\rho'})))}.$$
        By Lemma~\ref{spin_char_value_inn_prod_Gntilde}, 
        $$\wtilde{\chi}_{\rho}(\sigma_{\rho'})=2^{\frac{l(\rho')-l(\rho)-d(\rho)}{2}}\langle {Q_\rho}, P_{\rho'}\rangle=2^{\frac{l(\rho')-l(\rho)-d(\rho)}{2}}\sum_{T\in \mathrm{SBST}(\rho,\mu)}\prod_{i=1}^{t}\varepsilon_T(i)\gamma_{f_T(i)}(g_i(\wtilde{\theta}(\sigma_{\rho'}))).$$
        Our proof is now complete.
    \end{proof}

    \begin{remark}
        We remark that Stembridge gave a character-theoretic proof for the recursive rule that he obtained for the values of the irreducible characters of $G\wr S_n$. Our proof given here can be carried out using Macdonald's symmetric function framework to obtain the same for the irreducible characters of $G\wr S_n$, thereby giving a symmetric function-theoretic proof.
    \end{remark}

\section{\texorpdfstring{$0$-$1$}{0-1} boundary sequence of shifted Young diagrams}\label{0-1-seq-shifted}
    The Murnaghan-Nakayama rule involves removing border strips from Young diagrams. One can interpret this removal procedure using the $0$-$1$ boundary sequence of a Young diagram (see \cite[
    Exercise 7.59]{stan}). In fact, for $r\in \mathbb{N}$, the notion of the $r$-core and $r$-quotient of a partition can be described entirely using this sequence (see \cite{rr}). This is equivalent to the $r$-abacus diagram of a partition which is commonly used to work with cores and quotients (see \cite[2.7]{jk}).  In this final section, we describe an analogous 0-1 boundary sequence for shifted diagrams and interpret the Morris recursion rule for $\wtilde{S}_n$ in terms of this sequence. Clearly, such an interpretation can be applied to $s$-partite shifted diagrams as well. Therefore, we can implement Theorem~\ref{main_theorem} in a computer algebra system such as \textsf{SageMath}, using these sequences. 

    \medskip
    
    Let $\lambda \in \DP_n$. We associate a $0-1$ boundary sequence to the shifted diagram $SD_\lambda$ of $\lambda$ in the following way: it traces the outward boundary of the diagram starting from the bottom-most box of the main diagonal and ending at the last box of the first row. Each horizontal and vertical step is labeled by $1$ and $0$ respectively. We denote this boundary sequence by $B(\lambda)$. For convenience, we can add infinitely many 1's at the end of the sequence. The following example illustrates this. 

    \begin{example} Let $\lambda=(7,5,3,2,1)\in \DP_{18}$. We label the boundary of $SD_{\lambda}$ in the following way:

    \medskip

    \begin{center}
    
    \begin{tikzpicture}[scale=0.5]
        \foreach \x in {0,...,6} \draw (\x,4) rectangle (\x+1,5);
        \foreach \x in {1,...,5} \draw (\x,3) rectangle (\x+1,4);
        \foreach \x in {2,...,4} \draw (\x,2) rectangle (\x+1,3);
        \foreach \x in {3,...,4} \draw (\x,1) rectangle (\x+1,2);
        \foreach \x in {4}       \draw (\x,0) rectangle (\x+1,1);

        \node[anchor=west, inner sep=1pt] at (5,0.5) {$0$};
        \node[anchor=west, inner sep=1pt] at (5,1.5) {$0$};
        \node[anchor=west, inner sep=1pt] at (5,2.5) {$0$};
        \node[anchor=west, inner sep=1pt] at (6,3.5) {$0$};
        \node[anchor=west, inner sep=1pt] at (7,4.5) {$0$};

        \node[anchor=north, inner sep=1pt] at (4.8,0) {$1$};
        \node[anchor=north, inner sep=1pt] at (5.8,3) {$1$};
        \node[anchor=north, inner sep=1pt] at (6.8,4) {$1$};
    \end{tikzpicture}

    \end{center}
    \noindent The corresponding boundary sequence $B(\lambda)=(1,0,0,0,1,0,1,0)$. 
    \end{example}
    
    \noindent This first $1$ is called the \textit{pivot} $1$ of $B(\lambda)$. Thus, if $\lambda$ is not an empty partition, then $B(\lambda)$ always starts with $1$ and ends with $0$ (after removing the unwanted trailing 1's). If $\lambda$ is empty, then $B(\lambda)$ is just a sequence of 1's. Conversely, given a boundary sequence $B$, one can uniquely recover the strict partition $\lambda$ for which $B=B(\lambda)$. Suppose that $B$ contains $l$ zeros occurring at positions $p_1,p_2,\ldots,p_l$ listed from left to right. For $1\le i\le l$, let $a_i$ denote the number of $1$'s between the pivot $1$ and the $0$ (including pivot $1$) at position $p_i$. Then the strict partition corresponding to the boundary sequence $B(\lambda)$ is $\lambda=\bigl(a_l+l-1, a_{l-1}+l-2, \ldots, a_2+1, a_1\bigr).$
    
    Now we interpret the Morris recursion rule via the boundary sequence. In order to do that, we simply interpret in terms of the boundary sequence the following: (1) the removal of a border strip of Type-I or Type-II or a double strip of Type-III of size $r$, where $r$ is an odd positive integer, and  (2) the constant $\varepsilon(\delta)$ corresponding to a removal $\delta$ of each type, both of which are explained in Subsection~\ref{Morris_recursion_rule_Sntilde}. Let $\lambda=(\lambda_1,\lambda_2,\cdots,\lambda_l)\in \DP_n$ and the corresponding boundary sequence $B(\lambda)=(b_1,b_2,\cdots, b_t)$, where $b_1=1$ is the pivot. Let $r$ be an odd positive integer. Let $\hat{\lambda}$ denote the partition obtained after removing a border strip or a double strip.

    \medskip

    \begin{enumerate}
        \item \textbf{Type-I}: Recall that this type corresponds to a removal of a border strip of size $r$ from $SD_{\lambda}$ (which does not end at the main diagonal) and is equivalent to subtracting $r$ from a part of $\lambda$ which is strictly greater than $r$ and rearranging the resulting partition in decreasing order. An existence of such a border strip of size $r$ is guaranteed if and only if there exists some $i\ (i\neq 1)$ such that $b_i=1$ and $b_{i+r}=0$.  The boundary sequence  obtained after flipping $1$ and $0$ in the $i$-th and $(i+r)$-th positions of $B(\lambda)$ is precisely $B(\hat{\lambda})$. It is clear that $\varepsilon(SD_{\lambda/\hat{\lambda}})$ is $(-1)^z$, where $z$ is the number of zeros in $B(\lambda)$ strictly lying between $b_i$ and $b_{i+r}$.
        
        \item \textbf{Type-II}: Recall that this type corresponds to a removal of a border strip of size $r$ from $SD_{\lambda}$, which ends at the main diagonal of the last row, and is equivalent to the removal of a part of $\lambda$ that equals $r$. An existence of such a border strip of size $r$ is guaranteed if and only if $b_{r+1}=0$.  The boundary sequence obtained after flipping $b_{r+1}$ to $1$ of $B(\lambda)$ is precisely $B(\hat{\lambda})$. It is clear that $\varepsilon(SD_{\lambda/\hat{\lambda}})$ is $(-1)^z$, where $z$ is the number of 0's in $B(\lambda)$ strictly lying between $b_1$ and $b_{r+1}$. 
        
        \item \textbf{Type-III}: Recall that this type corresponds to a removal of a double strip of size $r$, that is, two consecutive border strip removals of size $x$ and $y$ ($x>y$ and $x+y=r$), both of which end at the main diagonal of the last row. This is equivalent to removing two parts from $\lambda$, one of which equals $x$ and the other equals $y$. An existence of such a double strip is guaranteed if and only if  $b_{x+1}=0$ and $b_{y+1}=0$. Flipping both these zeros to one (that is, $b_{x+1}=b_{y+1}=1$) yields the boundary sequence of $\hat{\lambda}$. In this case, $\varepsilon(SD_{\lambda/\hat{\lambda}})=2(-1)^{z}$, where $z$ is the number of zeroes strictly between $b_1$ and $b_{x+1}$ of $B({\lambda})$.
    \end{enumerate}

    \begin{example}
    Let $\lambda=(7,5,3,2,1)$. The following diagram illustrates the removal of a border strip or a double strip of size 3 at each step, using the 0-1 boundary sequence of the shifted diagram we obtain at each step. 
    \begin{center}
    \renewcommand{\arraystretch}{1.6}
    \small

    \begin{tabular}{ccc}
        $\lambda^1=\lambda=(7,5,3,2,1)$ & & $\lambda^2=(7,5,2,1)$\\
        \bardiag{5}{7,5,3,2,1}{4/3,4/4,4/5}
        & $\arr{Type-II removal}$ &
        \bardiag{4}{7,5,2,1}{5/1,6/1,5/2}\\
        $B(\lambda^1)=(1,0,0,\bx{0},1,0,1,0)$
        & $\arr{flip the boxed $0$ to $1$}$ &
        \tm{d2}{$B(\lambda^2)=(1,0,0,1,\bx{1},0,1,\bx{0})$}
    \end{tabular}

    \vspace{1.4cm}

    \begin{tabular}{ccc}
        \tm{t3}{$\lambda^3=(5,4,2,1)$} & & $\lambda^4=(5,4)$\\
        \bardiag{4}{5,4,2,1}{2/3,3/3,3/4}
        & $\arr{Type-III removal}$ &
        \bardiag{2}{5,4}{4/1,3/2,4/2}\\
        $B(\lambda^3)=(1,\bx{0},\bx{0},1,0,0,1,1)$
        & $\arr{flip the boxed zeros to one}$ &
        \tm{d4}{$B(\lambda^4)=(1,1,\bx{1},1,0,\bx{0},1,1)$}
        \end{tabular}
        \flowarrow{d2}{t3}{Type-I removal}{flip the boxed $1$ to $0$ and the boxed $0$ to $1$}

        \vspace{1.4cm}

        \begin{tabular}{ccc}
        \tm{t5}{$\lambda^5=(4,2)$} & & $\lambda^6=(2,1)$\\
        \bardiag{2}{4,2}{2/1,3/1,2/2} & $\arr{Type-I removal}$ &
        \bardiag{2}{2,1}{0/1,1/1,1/2}\\
        $B(\lambda^5)=(1,\bx{1},0,1,\bx{0},1,1,1)$ & $\arr{flip the boxed $1$ to $0$ and the boxed $0$ to $1$}$ & \tm{d6}{$B(\lambda^6)=(1,\bx{0},\bx{0},1,1,1,1,1)$}
    \end{tabular}
    \flowarrow{d4}{t5}{Type-I removal}{flip the boxed $1$ to $0$ and the boxed $0$ to $1$}

    \vspace{1.4cm}

    \begin{tabular}{c}
        \tm{t7}{$\lambda^7=\varnothing$}\\
        $B(\lambda^7)=(1,1,1,1,1,1,1,1)$.
    \end{tabular}
    \flowarrow{d6}{t7}{Type-III removal}{flip the boxed zeros to one}
    \end{center}
    Thus, we obtain the empty diagram in the last step, whose 0-1 boundary sequence is all 1's. 
    \end{example}

    \begin{remark}
        Morris used the term $r$-bar to refer collectively to a border strip or a double strip of size $r$. Analogous to $r$-core and $r$-quotient of a partition, Morris and Yasheen defined $r$-bar-core and $r$-bar-quotient of a strict partition (see \cite{mo.86}). They defined the notion of a bar-abacus diagram to interpret the removal of bars from a shifted diagram of a strict partition and used it to prove several results on bar-core and bar-quotient of a strict partition. These results are similar to those available for cores and quotients of a partition (see \cite[2.7]{jk}). In fact, the example given above demonstrates one possible way of obtaining the $3$-bar-core of $(7,5,3,2,1)$, which is the empty partition. We believe that all these results can be re-proved using the 0-1 boundary sequence defined above. We shall not pursue it here. 
    \end{remark}
\subsection*{Acknowledgments} We thank Prof. Arvind Ayyer for several conversations on this work. The first-named author gratefully acknowledges the financial support received from the New Faculty Seed Grant (NFSG), BITS Pilani, under Grant No. N7/26/805. The second-named author thanks the support of the Institute Postdoctoral Fellowship of IISER Mohali during this work. 
\bibliographystyle{amsalpha} 
\bibliography{Biblography}
\end{document}